\documentclass[12pt]{article}
\usepackage{amssymb,amsthm,amsmath,latexsym}
\newtheorem{thm}{Theorem}[section]
\newtheorem{prop}[thm]{Proposition}
\newtheorem{lem}[thm]{Lemma}
\newtheorem{cor}[thm]{Corollary}
\theoremstyle{remark}
\newtheorem{rem}[thm]{Remark}
\newcommand{\FF}{\mathbb{F}}
\newcommand{\ZZ}{\mathbb{Z}}
\newcommand{\RR}{\mathbb{R}}

\newcommand{\cC}{\mathcal{C}}
\newcommand{\cD}{\mathcal{D}}
\newcommand{\allone}{\mathbf{1}}
\DeclareMathOperator{\Aut}{Aut}
\DeclareMathOperator{\wt}{wt}
\DeclareMathOperator{\supp}{supp}

\begin{document}
\title{Some Extremal Self-Dual Codes and Unimodular Lattices
in Dimension $40$\footnote{This work was supported by JST PRESTO program.}}
%% 1: 2012/2/8  (H)
%% 2: 2012/2/20 (H)
%% 3: 2012/2/28 (BB)
%% 3.1: 2012/3/07 (H)
%% 4: 2012/3/08 (H)
%% 5: 2012/3/15 (BB)
%% 6: 2012/3/18 (H)
%% 6.1: 2012/3/18 (H)
%% 6.2: 2012/3/21 (BB)
%% 6.3: 2012/3/22 (H)

\author{
Stefka Bouyuklieva\thanks{Faculty of Mathematics and Informatics,
Veliko Tarnovo University, 5000 Veliko Tarnovo, Bulgaria},
Iliya Bouyukliev\thanks{Institute of Mathematics and Informatics,
Bulgarian Academy of Sciences, P.O. Box 323, Veliko Tarnovo,
Bulgaria}
and
Masaaki Harada\thanks{Department of Mathematical Sciences,
Yamagata University,
Yamagata 990--8560, Japan, and
PRESTO, Japan Science and Technology Agency (JST), Kawaguchi,
Saitama 332--0012, Japan.
email: mharada@sci.kj.yamagata-u.ac.jp}
}

\maketitle

\begin{abstract}
In this paper, binary extremal singly even self-dual codes of length $40$
and extremal odd unimodular lattices in dimension $40$ are studied.
We give a classification of extremal singly even
self-dual codes of length $40$.
We also give a classification of
extremal odd unimodular lattices
in dimension $40$ with shadows having $80$ vectors of norm $2$
through their relationships with
extremal doubly even self-dual codes of length $40$.
%We also discuss related extremal self-dual $\ZZ_4$-codes
%and extremal even self-dual additive $\FF_4$-codes.
\end{abstract}

%%%%%%%%%%%%%%%%%%%%%%%%%%%%%%%%%%
\section{Introduction}

Self-dual codes and unimodular lattices are studied  from several
viewpoints (see~\cite{SPLAG} for an extensive bibliography). Many
relationships between self-dual codes and unimodular lattices are
known and therefore these two objects have similar properties.
%there are similar situations between two subjects.
In this paper, binary
singly even self-dual codes of length $40$ and odd unimodular
lattices in dimension $40$ are studied.
%using relationships with
%binary doubly even self-dual codes of length $40$.
%This study gives new examples of
%similar situations between self-dual codes and unimodular lattices.
Both binary self-dual codes and unimodular lattices
are divided into two classes, namely,
doubly even self-dual codes and singly even self-dual codes, and
even unimodular lattices and odd unimodular lattices.
In addition,
a binary doubly even self-dual code of length $n$
(as well as an even unimodular lattice in dimension $n)$
exists if and only if $n \equiv 0 \pmod 8$.
This motivates to study binary self-dual codes
of length $n$ and  unimodular lattices in dimension $n$
when $n \equiv 0 \pmod 8$.

It is a fundamental problem to classify self-dual codes and
unimodular lattices. Much work has been done towards classifying
binary self-dual codes and unimodular lattices for modest lengths
and dimensions (see~\cite{SPLAG}). %For example, a classification
%of binary doubly even self-dual codes was done for lengths up to
%$32$~\cite{CPS},  \cite{Pless72} and \cite{PS75}. Recently, the
%classification has been extended to length
%$40$~\cite{BHM40}.
A big progress has been realized recently in the classification of
binary self-dual codes with different parameters. Five papers
related to this subject were presented in 2010 and 2011. All
self-dual codes of length 36 were classified in \cite{HM36}. The
extremal codes among them, namely the self-dual $[36,18,8]$ codes,
had been previously classified in \cite{n=36}, so the result was
approved in \cite{HM36}. The classification of the extremal
self-dual $[38,19,8]$ codes was done by three research groups
independently. In \cite{AGKSS38} the authors proved that there are
exactly 2744 extremal self-dual $[38, 19, 8]$ codes, two
$s$-extremal self-dual $[38, 19, 6]$ codes, and 1730 $s$-extremal
self-dual $[38,19,8]$ codes, up to equivalence.
Recently, a classification of doubly
even self-dual codes of length $40$ has been done in~\cite{BHM40},
namely there are $94343$  doubly even self-dual codes, $16470$ of
which are extremal, up to equivalence. Using these codes, the
authors classified the extremal self-dual $[38,19,8]$ codes. In
\cite{BB38}, an efficient algorithm for classification of binary
self-dual codes was described  and then applied to length 38. In
this way all self-dual codes of length 38 were classified.

%%%%%%%%%%%%%%%%%%%
The main aim of this paper is to give a classification of binary
extremal singly even self-dual codes of length $40$. Some part of
the classification is done by considering a relationship with
binary doubly even self-dual codes of the same length. In
addition, the classification is completed using the algorithm
presented in \cite{BB38}. A partial classification of extremal odd
unimodular lattices in dimension $40$ is also established from the
classification of binary extremal doubly even self-dual codes of
length $40$.

%%%%%%%%%%%%%%%%%%%
% {\bf (this part should be re-written later)}
This paper is organized as follows. In Section~\ref{Sec:Def}, we
give definitions and some basic properties of the self-dual codes
and unimodular lattices used in this paper.
%%%%
In Section~\ref{Sec:C}, we study the binary extremal singly even
self-dual codes of length $40$. The number of vectors of weight
$4$ in the shadow of a binary extremal singly even self-dual code
of length $40$ is at most $10$ \cite{C-S}. We demonstrate that
there are $19$ inequivalent binary extremal singly even self-dual
codes with shadows having $10$ vectors of weight $4$. This
classification is done by considering a relationship with binary
doubly even self-dual codes of length $40$ such that the supports
of all codewords of weight $4$ are $T$-decompositions. We mention
a relationship between such binary  doubly even self-dual codes
and even self-dual additive $\FF_4$-codes. By a method similar to
the above, our classification is also extended to all cases that
shadows have minimum weight $4$.
% In order to give a classification of binary
% extremal singly even self-dual codes  with shadows of minimum
% weight $8$,
% we employ another classification method.
% Using this method, we complete a classification of
% binary
% extremal singly even self-dual codes of length $40$.
The classification of binary extremal (singly even) self-dual codes
of length $40$ is also completed by an approach which does
not depend on the weight enumerators. As a consequence,
we have a classification of binary extremal singly even self-dual codes
with shadows of minimum weight $8$, which is the remaining case.
More precisely,
combined with the above classification,
we demonstrate that there are $10200655$
inequivalent binary
extremal singly even self-dual codes of length $40$.
% This shows that $40$ is the smallest length $n$ such that
% there are many more binary singly even self-dual codes
% than doubly even self-dual codes for
% length $n$ and minimum weight $4\lfloor n/24 \rfloor +4$,
% up to equivalence.
%%%%
In Section~\ref{Sec:L}, for extremal odd unimodular lattices,
we consider a situation which is
similar to that for binary extremal singly even self-dual codes
with ten vectors of weight $4$ in the shadows, given
in Section~\ref{Sec:C}.
The number of vectors of norm $2$
in the shadow of an extremal odd unimodular lattice in dimension
$40$ is at most $80$.
It is shown that
there are $16470$ non-isomorphic extremal odd unimodular
lattices in dimension $40$
with shadows having $80$ vectors of norm $2$.
This classification is done by considering a relationship with binary
extremal doubly even self-dual codes of length $40$.
Finally, in Section~\ref{Sec:Z4},
we investigate theta series for which
there is an extremal odd unimodular lattice in dimension $40$.
The existence of a number of
extremal self-dual $\ZZ_4$-codes of length $40$ is also given.

Generator matrices of all inequivalent binary extremal singly even
self-dual codes of length $40$ can be obtained electronically
from~\cite{Data}. Most of the computer calculations in this paper
were done by {\sc Magma}~\cite{Magma}. In Section
\ref{Subsection:0}, the package \textsc{Self-dual-bin} by the
second author is also used, where it's main algorithm implemented
in the program \textsc{Gen-self-dual-bin} is described in
\cite{BB38}.

%%%%%%%%%%%%%%%%%%%%%%%%%%%%%%%%%%
\section{Preliminaries}
\label{Sec:Def}

\subsection{Binary self-dual codes}
Let $\FF_q$ denote the finite field of order $q$, where $q$ is a
prime power.
Any subset of $\FF_q^n$ is called an $\FF_q$-code of length $n$.
An $\FF_q$-code is linear if it is a linear subspace of $\FF_q^n$.
An $\FF_2$-code is called {\em binary} and
all codes in this paper mean binary linear codes unless otherwise noted.

A code is called {\em doubly even} if all codewords have weight
$\equiv 0 \pmod 4$.
The \textit{dual code} $C^{\perp}$ of a code
$C$ of length $n$ is defined as
$
C^{\perp}=
\{x \in \FF_2^n \mid x \cdot y = 0 \text{ for all } y \in C\},
$
where $x \cdot y$ is the standard inner product.
A code $C$ is called
% \textit{self-orthogonal} if $C \subset C^{\perp}$,
% and $C$ is  called
\textit{self-dual} if $C = C^{\perp}$.
A self-dual code which is not doubly even
is called
{\em singly even}.
%if some codeword has weight $\equiv 2 \pmod 4$.
A doubly even self-dual code of length $n$ exists if and
only if $n \equiv 0 \pmod 8$, while a singly even self-dual code
of length $n$ exists if and
only if $n$ is even.

%%%%%%%%%%%%%%
Let $C$ be a singly even self-dual code and let $C_0$ denote the
subcode of codewords having weight $\equiv0\pmod4$. Then $C_0$ is
a subcode of codimension $1$. The {\em shadow} $S$ of $C$ is
defined to be $C_0^\perp \setminus C$. There are cosets
$C_1,C_2,C_3$ of $C_0$ such that $C_0^\perp = C_0 \cup C_1 \cup
C_2 \cup C_3 $, where $C = C_0  \cup C_2$ and $S = C_1 \cup C_3$.
Shadows were introduced by Conway and Sloane~\cite{C-S}, in order
to
%derive new upper bounds for the minimum weight of
%singly even self-dual codes, and to
provide restrictions on the weight enumerators of singly even
self-dual codes. Two self-dual codes $C$ and $C'$ of length $n$
are said to be {\em neighbors} if $\dim(C \cap C')=n/2-1$. Since
every self-dual code $C$ of length $n$ contains the all-one vector
$\allone$, $C$ has $2^{n/2-1}-1$ subcodes $D$ of codimension $1$
containing $\allone$. Since $\dim(D^\perp/D)=2$, there are two
self-dual codes rather than $C$ lying between $D^\perp$ and $D$.
When $C$ is doubly even, one of them is doubly even and the other
is singly even for each subcode $D$. If $C$ is a singly even
self-dual code of length divisible by $8$, then $C$ has two doubly
even self-dual neighbors, namely $C_0 \cup C_1$ and $C_0 \cup C_3$
(see~\cite{BP}).

%{\bf (Please put necessary definitions for your part!)}

The minimum weight $d(C)$ of a self-dual code $C$ of length
$n$ is bounded by $d(C)\le 4 \lfloor n/24 \rfloor +4$ unless $n
\equiv 22 \pmod{24}$ when $d(C) \le 4 \lfloor n/24 \rfloor+6$~\cite{Rains}.
%~\cite{Mallows-Sloane} and \cite{Rains}.
% In addition, a self-dual
% code of length $24k$ and minimum weight $4k+4$ is doubly
% even~\cite{Rains}. %Hence, a singly even self-dual code $C$ of
% length $24k$ satisfies $d(C) \le 4k+2$.
We say that a self-dual
code meeting the upper bound is {\em extremal}.

%%%%%%%%%%%%%%
Two codes $C$ and $C'$ are {\em equivalent}, denoted $C
\cong C'$,
%% Two codes are {\em equivalent}
if one can be obtained from the other by permuting the
coordinates. An {\em automorphism} of $C$ is a
permutation of the coordinates which preserves the code. The set
consisting of all automorphisms of $C$ forms a group called the
\emph{automorphism group} of this code and it is denoted by
$\Aut(C)$. The number of the codes equivalent to $C$ is
$n!/\#\Aut(C)$, where $n$ is the length of $C$. If we consider
the action of the symmetric group $S_n$ on the set $\Omega_n$ of
all self-dual codes of length $n$, then two codes from this set are
equivalent if they belong to the same orbit. This action induces
an equivalence relation in $\Omega_n$ and the equivalence classes
are the orbits with respect to the action of $S_n$. To classify
the self-dual codes of length $n$ means to find exactly one
representative of each equivalence class.

% $$*** - We \ suggest \ to \ be \ removed$$
% A classification of doubly even self-dual codes was done for
% lengths $8,16$ in~\cite{Pless72}, for length $24$ in~\cite{PS75}
% and for length $32$ in~\cite{CPS}. Recently, a classification of
% doubly even self-dual codes of length $40$ has been done
% in~\cite{BHM40}, namely there are $94343$  doubly even self-dual
% codes, $16470$ of which are extremal, up to equivalence.
% %%%
% For singly even self-dual codes,
% a classification is known for lengths up to $34$~\cite{B06},
% \cite{BR02}, \cite{CPS}, \cite{Pless72} and \cite{PS75}.
% Recently,
% the classification has been extended to length $36$~\cite{HM36}
% and length $38$~\cite{BB38}.
% $$*** $$
%
% {\bf It is OK to remove here, but this paper does not
% contain any information on classification of self-dual codes
% of lengths up to 34.  Is it OK?  Will we mention in Section 1?}

%%%%%%%%%%%%%%%%%%%%%
\subsection{Self-dual additive $\FF_4$-codes and $\ZZ_4$-codes}
%{\bf (Here is minor changed)}

In Section \ref{Subsec:C}, we also consider additive $\FF_4$-codes
where $\FF_4=\{ 0,1, \omega, \omega^2\}$. Such a code is a
$k$-dimensional $\FF_2$-subspace of $\FF_4^n$ and so has $2^k$
codewords.
An additive $\FF_4$-code is even if all its codewords have even
Hamming weights.
Two additive $\FF_4$-codes $\cC_1$ and $\cC_2$ are equivalent
if there is a map from $S_3^n\times S_n$ sending $\cC_1$ onto $\cC_2$,
where $S_n$ acts on the set of the $n$ coordinates and $S_3$
permutes the elements $1, \omega, \omega^2$ of the field. The
automorphism group of $\cC$, denoted by $\Aut(\cC)$, consists of all
elements of $S_3^n\times S_n$ which preserve the code.

An additive $\FF_4$-code $\cC$ is called self-dual if
$\cC = {\cC}^*$, where
the dual code $\cC^*$ of $\cC$ is defined as
$\{x \in \FF_4^n \mid x * y = 0 \text{ for all } y \in \cC\}$
under
$$x * y=\sum_{i=1}^n (x_iy_i^2+x_i^2y_i) \ \
\mbox{for} \ \ x=(x_1,\ldots,x_n), y=(y_1,\ldots,y_n) \in
\FF_4^n.$$
%It is more useful to consider \emph{trace inner
%product} for these codes, given by
%$$x\cdot y=\sum_{i=1}^n (x_iy_i^2+x_i^2y_i), \ \ x,y\in\FF_4^n.$$
Even self-dual additive $\FF_4$-codes exist only in even lengths.
The minimum Hamming weight $d(\cC)$ of an even self-dual additive
$\FF_4$-code $\cC$ of length $n$ is bounded by $d(\cC)\le 2
\lfloor n/6 \rfloor +2$ \cite[Theorem 33]{RS-Handbook}. We say
that an even self-dual additive $\FF_4$-code meeting the upper
bound is {extremal}.

The last family of self-dual codes which we consider is the set of
self-dual $\ZZ_4$-codes, where
$\ZZ_4$ denotes the ring of integers modulo $4$.
The self-dual $\ZZ_4$-codes are
connected with unimodular lattices of a special form \cite{Z4-BSBM}.
A $\ZZ_{4}$-code $\cC$ of length $n$
is a $\ZZ_{4}$-submodule of $\ZZ_{4}^n$.
A $\ZZ_4$-code $\cC$ is {self-dual} if $\cC={\cC}^\perp$,
where the {dual} code ${\cC}^\perp$ of $\cC$ is defined as
$\{ x \in \ZZ_{4}^n \mid  x \cdot y = 0$ for all $y \in \cC\}$
under the standard inner product $x \cdot y$.
We say that two $\ZZ_4$-codes are {equivalent}
if one can be obtained from the other by permuting the
coordinates and (if necessary) changing the signs of certain
coordinates.

The {Euclidean weight} of a codeword $x=(x_1,\ldots,x_n)$ of $\cC$ is
$n_1(x)+4n_2(x)+n_3(x)$, where $n_{\alpha}(x)$ denotes
the number of components $i$ with $x_i=\alpha$ $(\alpha=1,2,3)$.
A $\ZZ_4$-code $\cC$ is {Type~II}
if $\cC$ is self-dual and
the Euclidean weights of all codewords of $\cC$ are divisible
by 8~\cite{Z4-BSBM,Z4-HSG}.
A self-dual code which is not Type~II is called {Type~I}.
A Type~II $\ZZ_4$-code of length $n$ exists if and
only if $n \equiv 0 \pmod 8$,
while a Type~I $\ZZ_4$-code exists for every length.
The minimum Euclidean
weight $d_E(\cC)$ of a Type~II $\ZZ_4$-code  $\cC$
of length $n$ is bounded by
$d_E(\cC) \le 8 \lfloor n/24 \rfloor +8$~\cite{Z4-BSBM}.
%%%%%
It was also shown in~\cite{RS-bound} that
the minimum Euclidean
weight $d_E(\cC)$ of a Type~I code $\cC$
of length $n$ is bounded by
$d_E(\cC) \le 8 \lfloor n/24 \rfloor +8$
if $n \not\equiv 23 \pmod{24}$, and
$d_E(\cC) \le 8 \lfloor n/24 \rfloor +12$
if $n \equiv 23 \pmod{24}$.
A self-dual $\ZZ_4$-code meeting the upper bound is called
{extremal}.

% Codes differing by only a permutation of coordinates
% are called permutation-equivalent.
% The automorphism group $\Aut(C)$
% of $C$ consists of all permutations and sign-changes of the
% coordinates which preserve $C$.

% The {\em minimum Euclidean weight} $d_E(C)$ of $C$ is the smallest Euclidean
% % weight among all nonzero codewords of $C$.

% Two $\ZZ_4$-codes are {\em equivalent} if one can be obtained from the
% other by permuting the coordinates and (if necessary) changing
% the signs of certain coordinates.

%%%%%%%%%%%%%%%%%%%%%%%%%%%%%%%%%%%%%%%%%%%%%%%%%%%%%%%%%%%%%%%%%%%%%%%
\subsection{Unimodular lattices}
A (Euclidean) lattice $L \subset \RR^n$
in dimension $n$
is {\em integral} if
$L \subset L^{*}$, where
the dual lattice $L^{*}$ of $L$ is defined as
$\{ x \in {\RR}^n \mid (x,y) \in \ZZ \text{ for all }
y \in L\}$ under the standard inner product $(x,y)$.
An integral lattice is called {\em even}
if the norm $(x,x)$ of every vector $x$ is even.
A lattice $L =L^*$ is called {\em unimodular}.
A unimodular lattice which is not even is called
{\em odd}.
An even unimodular lattice in dimension $n$ exists if and
only if $n \equiv 0 \pmod 8$, while
an odd  unimodular lattice exists for every dimension.
Two lattices $L$ and $L'$ are {\em isomorphic}, denoted $L \cong L'$,
if there exists an orthogonal matrix $A$ with
$L' = L \cdot A$.
%The automorphism group $\Aut(L)$ of $L$ is the group of all
%orthogonal matrices $A$ with $L = L \cdot A$.

%%%%%%%%%%%%%%
Rains and Sloane~\cite{RS-bound} showed that
the minimum norm $\min(L)$ of a unimodular
lattice $L$ in dimension $n$
is bounded by
$\min(L) \le 2 \lfloor n/24 \rfloor+2$
unless $n=23$ when $\min(L) \le 3$.
% Gaulter~\cite{Gaulter} showed that
% any unimodular lattice in dimension $24k$
% meeting the upper
% bound has to be even, which was conjectured by
% Rains and Sloane.
% Hence,  an odd unimodular lattice $L$
% in dimension $24k$
% satisfies $\min(L) \le 2k+1$.
We say that a unimodular lattice meeting the upper
bound is {\em extremal}.
%% For dimension $48$,
%% three non-isometric extremal even unimodular lattices
%% are known and these lattices are denoted by $P_{48p}$,
%% $P_{48q}$ and $P_{48n}$ (cf.\~\cite[p.\ xli]{SPLAG}).

%%%%%%%%%%%%%%
Let $L$ be an odd unimodular lattice and let $L_0$ denote its
sublattice of vectors of even norms.
Then $L_0$ is a sublattice of $L$ of index $2$~\cite{CS-odd}.
The {\em shadow} $S$ of $L$ is defined to be $L_0^* \setminus L$.
There are cosets $L_1,L_2,L_3$ of $L_0$ such that
$L_0^* = L_0 \cup L_1 \cup L_2 \cup L_3$, where
$L = L_0  \cup L_2$ and $S = L_1 \cup L_3$.
Shadows for odd unimodular lattices appeared in~\cite{CS-odd}
and also in~\cite[p.~440]{SPLAG},
in order to
% derive new upper bounds for the minimum norm of
% odd unimodular lattices, and to
provide
restrictions on the theta series of odd unimodular lattices.
%%%%%%%%%%%%%%
Two lattices $L$ and $L'$ are {\em neighbors} if
both lattices contain a sublattice of index $2$
in common.
If $L$ is an odd unimodular lattice in dimension divisible by
$8$, then $L$ has two even unimodular neighbors
of $L$, namely, $L_0 \cup L_1$ and $L_0 \cup L_3$.

%%%%%%%%%%%%%%%%%%%%%%%%%%%%%%%%%%%%%%%%%%%%%%%%%%%%%%%%%%%%%%%%%%
\section{Singly even self-dual codes of length 40}
\label{Sec:C}
In this section, we give a classification of extremal
singly even self-dual codes of length $40$.
For the case that the shadows have minimum weight $4$,
the classification was done in two different ways.
% It is known that
% the number of vectors of weight $4$
% in the shadow of an extremal singly even self-dual code
% of length $40$ is at most $10$~\cite{C-S}.
% We give a classification of extremal singly even self-dual codes
% with shadows having $10$ vectors of weight $4$
% through $T$-decompositions.
% Our classification is also extended to
% all cases that shadows have minimum weight $4$.

%%%%%
\subsection{Weight enumerators}
An extremal singly even self-dual code $C$ of
length $40$ and its shadow $S$
have the following weight enumerators:
\[
\left\{\begin{array}{cl}
W_{40,C,\beta}=&1
+ (125 + 16 \beta) y^8
+ (1664  - 64 \beta) y^{10}
+ (10720 + 32 \beta) y^{12} + \cdots,
\\
W_{40,S,\beta}=&\beta  y^4
+ (320  - 8 \beta)  y^8
+ (21120  +  28 \beta) y^{12} + \cdots,
\end{array}\right.
\]
respectively,
where $\beta$ is an integer with $0 \le \beta \le 10$~\cite{C-S}.
It was shown in~\cite{HM06} that an extremal singly even self-dual code
with weight enumerator $W_{40,C,\beta}$ exists if and only if
$\beta=0,1,\ldots,8,10$.

\begin{lem}\label{lem:1}
Let $C$ be an
extremal singly even self-dual code of length $40$
with weight enumerator $W_{40,C,\beta}$.
Then one of $C_0 \cup C_1$ and $C_0 \cup C_3$ is an
extremal doubly even self-dual code of length $40$
and the remaining one is a
doubly even self-dual code of length $40$ containing
$\beta$ codewords of weight $4$.
\end{lem}
\begin{proof}
The codes $C_0 \cup C_1$ and $C_0 \cup C_3$ are doubly even
self-dual codes~\cite{BP}.
The lemma is trivial in the cases
$\beta=0$  and $\beta=1$. Suppose that $S$ has at least two
vectors of weight $4$. Let $x,y$ be vectors of weight $4$ in $S$
with $x \ne y$. If $x \in C_1$ and $y \in C_3$, then $x+y \in C_2$
and $\wt(x+y) \le 6$, which contradicts the minimum weight of
$C_2$. Hence, we may assume without loss of generality that all
vectors of weight $4$ in $S$ are contained in $C_1$. Then $C_0
\cup C_3$ is extremal and $C_0 \cup C_1$ has $\beta$ codewords of
weight $4$.
\end{proof}

%% \begin{rem}
%% If $C$ is an extremal singly even self-dual code of length $40$
%% with shadow $S$ of minimum weight $8$
%% then both $C_0 \cup C_1$ and $C_0 \cup C_3$ are
%% extremal doubly even self-dual codes of length $40$
%% with covering radius $8$~\cite[p.~1225]{BP}.
%% \end{rem}

%%%
\subsection{Weight enumerator $W_{40,C,10}$,
$T$-decompositions and self-dual additive $\FF_4$-codes}\label{Subsec:C}

Let $C$ be a code of length $n \equiv 0 \pmod 4$.
A partition $\{T_1,T_2,\ldots,T_{\frac{n}{4}}\}$ of
$\{1,2,\ldots,n\}$
is called a {\em $T$-decomposition}
of $C$ if the following conditions hold:
\begin{align}
\label{eq:T1}
&T_1 \cup T_2 \cup \cdots \cup T_{\frac{n}{4}} = \{1,2,\ldots,n\}, \\
\label{eq:T2}
&\# T_i=4, \\
\label{eq:T3}
&T_i \cup T_j \text{ is the support of a codeword of $C$},
\end{align}
for
$i,j=1,2,\ldots,\frac{n}{4}$ and $i \ne j$~\cite{KKM}.
In particular, when all $T_i$ are the supports of codewords
of $C$, we say that $C$ contains the $T$-decomposition.

In this subsection, we assume that $C$ is an extremal singly even
self-dual code of length $40$ with weight enumerator
$W_{40,C,10}$. By Lemma~\ref{lem:1}, we may suppose without loss
of generality that all ten vectors of weight $4$ in $S$ are
contained in $C_1$.

\begin{lem}\label{lem:S4}
There is a $T$-decomposition for the self-dual codes $C$, $C_0
\cup C_1$ and $C_0 \cup C_3$. In particular, $C_0 \cup C_1$
contains a $T$-decomposition.
\end{lem}
\begin{proof}
Denote by $S_4$ the set
of the ten vectors of weight $4$ in the shadow $S$ of $C$. The
supports of any two vectors of $S_4$ are disjoint. Thus,
$\{\supp(x) \mid x \in S_4\}$ is a $T$-decomposition for the codes
$C$, $C_0 \cup C_1$ and $C_0 \cup C_3$.
\end{proof}

%% By Lemmas~\ref{lem:1} and \ref{lem:S4},
%% every extremal singly even
%% self-dual code $C$ of length $40$ with weight enumerator
%% $W_{40,C,10}$ is constructed from some doubly even self-dual
%% code  $D$
%% % with minimum weight $4$ containing a
%% % unique $T$-decomposition,
%% %such that the supports of the codewords of weight $4$
%% %are a $T$-decomposition,
%% satisfying the following condition
%% \begin{align}
%% \text{the supports of all codewords of weight $4$
%% are a $T$-decomposition},
%% \label{eq:Tdec}
%% \end{align}
%% as a neighbor.
%% It follows that $\dim(C \cap D)=19$ and $\allone \in C \cap D$.
%% % Note that every doubly even self-dual code of length $40$ contains
%% % $2^{19}-1$ subcodes of codimension $1$ containing $\allone$.
%% % For each subcode,
%% % there are three self-dual codes containing the subcode
%% % and two are doubly even
%% % and the other is singly even.

%%%%% AAA
By Lemmas~\ref{lem:1} and \ref{lem:S4}, every extremal singly even
self-dual code $C$ of length $40$ with weight enumerator
$W_{40,C,10}$ is constructed as a neighbor of a doubly even
self-dual code  $D$
% with minimum weight $4$ containing a
% unique $T$-decomposition,
%such that the supports of the codewords of weight $4$
%are a $T$-decomposition,
satisfying the following condition
\begin{align}
\text{the supports of all codewords of weight $4$ are a
$T$-decomposition}. \label{eq:Tdec}
\end{align}
It follows that $\dim(C \cap D)=19$ and $\allone
\in C \cap D$. Since $\dim(C_0^\perp/C_0)=2$, there are two doubly
even self-dual codes lying between $C_0^\perp$ and $C_0$. In this
case, one doubly even self-dual code $C_0 \cup C_1$ has minimum
weight $4$ and the other one $C_0 \cup C_3$ has minimum weight
$8$. Hence, $C_0$ is contained in a unique doubly even self-dual
code satisfying the condition (\ref{eq:Tdec}). This means that
extremal singly even self-dual codes, constructed in this way from
inequivalent doubly even self-dual codes satisfying
(\ref{eq:Tdec}), are inequivalent. In other words, it is
sufficient to check equivalences only for the subcodes $D$ with
$\dim(D)=19$, $d(D)=8$ and $\allone \in D$ contained in each
doubly even self-dual code satisfying (\ref{eq:Tdec}), in order to
complete our classification. This observation
substantially reduces the necessary calculations for
equivalence tests in our classification.

We start the classification of extremal singly even self-dual
codes with weight enumerator $W_{40,C,10}$ by investigating the
doubly even $[40,20,4]$ codes. There are $1093$ inequivalent
doubly even self-dual codes with $A_4=10$, where $A_i$ denotes the
number of codewords of weight $i$~\cite{BHM40}. We verified that
$19$ codes of them satisfy the condition (\ref{eq:Tdec}).
%such that the supports of the codewords of weight $4$
%are a $T$-decomposition.
% containing a unique $T$-decomposition.
%M The set of codewords of weight $4$ was obtained by
%M the {\sc Magma} function {\tt Words}.
%%%%
Note that every doubly even self-dual code of length $40$ contains
$2^{19}-1$ subcodes of codimension $1$ containing $\allone$.
%% For each of the $19$ codes, we verified that
%% $2^{10}$ subcodes have minimum weight $8$.
%% The minimum weight was verified by the {\sc Magma} functions
%% {\tt MinimumWeight}.
%% In addition, these $2^{10}$ codes are equivalent for each case.
%% This was done by the {\sc Magma} function {\tt IsIsomorphic}.
For each of the $19$ codes, we verified that
% all subcodes with minimum weight $8$ containing $\allone$
all such subcodes are equivalent.
%M This was done by the {\sc Magma} function {\tt IsIsomorphic}.
%%%
%%%
The calculations took about 65 minutes
using a single core of a PC Intel i7 6 core processor.
By the above observation in the previous paragraph,
we have the following:

\begin{prop}\label{prop:b10}
There are $19$ inequivalent extremal
singly even self-dual codes of length $40$ with weight enumerator
$W_{40,C,10}$.
\end{prop}

\begin{rem}\label{rem:2432}
A similar argument can be used for the classification of extremal
singly even self-dual codes of length $32$. The codes C$67$, C$68$
and C$69$ in~\cite[Table A]{CPS} are the only doubly even
self-dual codes of length $32$
%satisfying (\ref{eq:T1}), (\ref{eq:T2}) and (\ref{eq:T3}).
satisfying (\ref{eq:Tdec}).
%such that the supports of the codewords of weight $4$
%are a $T$-decomposition.
%containing a unique $T$-decomposition.
Note that the shadow of
an extremal singly even self-dual code
has exactly $8$ vectors of  weight $4$
for length $32$.
\end{rem}

%%% CR
Now we give some properties of the extremal singly even self-dual
codes of length $40$ with weight enumerator $W_{40,C,10}$. The
{\em covering radius} of a code $C$ of length $n$ is the smallest
integer $R$ such that spheres of radius $R$ around codewords of
$C$ cover the space $\FF_2^n$.
% It is known that
% the covering radius is the same as the largest value
% among weights of cosets.
% Here the weight of a coset is the smallest weight of
% a vector in the coset.
The covering radius is a basic and important geometric parameter
of a code. Let $R_{40}$ be the covering radius of an extremal
singly even self-dual code of length $40$. Then, by the
sphere-covering bound  and the Delsarte bound (see~\cite{A-P}), $6
\le R_{40} \le 14$.
%The set consisting of all permutations of the coordinates of $C$
%which preserves $C$, is called the
%{\em automorphism group} of $C$.
%% and it is denoted by $\Aut(C)$.
In Table~\ref{Tab:AutCR10},
we list
the number $N(\#\Aut,R)$ of extremal singly even self-dual codes
with automorphism groups of order $\#\Aut$
and covering radii $R$ for $\beta=10$.
%M The automorphism group $\Aut(C)$ and the covering radius $R$
%M were calculated by the {\sc Magma} functions {\tt AutomorphismGroup}
%M and {\tt CoveringRadius}, respectively.

%%%%%%%%%%%%%%%%%%%%%%%%%%%%%%%%%%%%%%%%%%%%%%%%
\begin{table}[thb]
\caption{$(\#\Aut,R,N(\#\Aut,R))$ for $\beta=10$}
\label{Tab:AutCR10}
\begin{center}
%{\small
{\footnotesize
%{\scriptsize
\begin{tabular}{ccccc}
\noalign{\hrule height0.8pt}
\multicolumn{5}{c}{$(\#\Aut,R,N(\#\Aut,R))$}\\
\hline
$(12288, 8, 1)$ &
$(16384, 8, 1)$ &
$(18432, 8, 1)$ &
$(20480, 8, 1)$ &
$(32768, 8, 1)$ \\
$(49152, 8, 3)$ &
$(65536, 8, 1)$ &
$(147456, 8, 1)$ &
$(245760, 8, 1)$ &
$(262144, 8, 1)$ \\
$(327680, 8, 1)$ &
$(737280, 8, 1)$ &
$(786432, 8, 1)$ &
$(1179648, 8, 1)$ &
$(1474560, 8, 1)$ \\
$(11796480, 8, 1)$ &
$(44236800, 7, 1)$ & \\
\noalign{\hrule height0.8pt}
  \end{tabular}
}
\end{center}
\end{table}
%%%%%%%%%%%%%%%%%%%%%%%%%%%%%%%%%%%%%%%%%%%%%%%%
From Table~\ref{Tab:AutCR10}, for $\beta=10$,
there is a unique extremal singly even self-dual code with
covering radius $7$.
This code must be equivalent to the extremal
singly even self-dual code with covering radius $7$
given in~\cite[Section 3]{HO}.
The coset weight distribution of
some extremal singly even self-dual code of length $40$
with covering radius $7$
was given in~\cite[Table IX]{Ozeki-S},
noting that the code was listed without giving its construction.
It is claimed in~\cite[Table IX]{Ozeki-S} that
there are 40000 cosets with weight enumerator
$4 y^6 + 152 y^8 + 1644 y^{10}  +  10608 y^{12} + \cdots$.
However,
we verified that this is incorrect and there
are only 14400 cosets with the weight enumerator
and there are 25600 cosets with the following
weight enumerator
\begin{multline*}
 4 y^{6}
+ 168 y^{8}
+ 1580 y^{10}
+ 10640 y^{12}
+ 44388 y^{14}
\\
+ 119768 y^{16}
+ 216172 y^{18}
+ 263136 y^{20}
+ \cdots + 4 y^{34}.
\end{multline*}
%M where the coset weight distributions
%M were calculated by the {\sc Magma} function {\tt WeightDistribution}.

We give a characterization of the set of codewords of minimum
weight in an extremal singly even self-dual code of length $40$
with weight enumerator $W_{40,C,10}$ in the following proposition.
Recall that a $t$-$(v,k,\lambda)$ design is a set $X$ of $v$
points together with a collection of $k$-subsets of $X$ (called
blocks) such that every $t$-subset of $X$ is contained in exactly
$\lambda$ blocks.

\begin{prop}\label{prop:Cmin}
Let $C$ be an extremal singly even self-dual code of length $40$
with weight enumerator $W_{40,C,10}$.
Then the set of the supports of codewords of weight $8$
in $C$ forms a $1$-$(40,8,57)$ design.
\end{prop}
\begin{proof}
By Lemma~\ref{lem:1}, $C_0 \cup C_3$ is an
extremal doubly even self-dual code.
By the
Assmus--Mattson theorem (see~\cite[Section~9]{RS-Handbook}),
the set of the supports of codewords of weight $8$ in
an extremal doubly even self-dual code of length $40$
forms a $1$-$(40,8,57)$ design.
Since the numbers of codewords of weight $8$ in
$C$ and $C_0 \cup C_3$ are identical,
the result follows.
\end{proof}

%%%%%%%%%%%%% even self-dual additive $\FF_4$-codes
Now we study a relationship between the doubly even self-dual
codes satisfying the condition (\ref{eq:Tdec})
%such that the supports of the codewords of weight $4$
%are a $T$-decomposition
%containing a unique $T$-decomposition
and the even self-dual additive $\FF_4$-codes.
% , where
% $\FF_4=\{0,1,\omega,\omega^2\}$ (see e.g.~\cite{BG00},
% \cite{CRSS}, \cite{DP06}, \cite{Hohn} for undefined terms
% concerning self-dual additive $\FF_4$-codes).
Let $\cC$ be an even
self-dual additive $\FF_4$-code of length $n$. From $\cC$, the
following code of length $4n$ is obtained:
\[
B(\cC)= \rho({\cC})+\{(0000),(1111)\}^n,
\]
where $\rho$ is the map $\FF_4^n \rightarrow \FF_2^{4n}$
induced from $\FF_4 \rightarrow \FF_2^4$,
$0 \mapsto (0000)$,
$1 \mapsto (1100)$,
$\omega \mapsto (1010)$ and
$\omega^2 \mapsto (0110)$.
It is easy to see that
if $\cC$ is an even self-dual additive $\FF_4$-code of length $n$,
then $B(\cC)$ is a doubly even self-dual code of length $4n$
and $A_4 \ge n$,
containing a $T$-decomposition.
%Moreover, if $\cC$ has no codeword of weight $2$ then
%$B(\cC)$ has a unique $T$-decomposition.
% correct?
Conversely, if a doubly even self-dual code $C$ contains a
$T$-decomposition, then
$C \cong B(\cC)$ for some even self-dual additive $\FF_4$-code
$\cC$~\cite{S}.
A classification of extremal even self-dual additive $\FF_4$-codes
of length $10$ was given in~\cite{BG00}.
There are $19$ such codes, up to equivalence~\cite[Table 2]{BG00}.
% As described above, there are $19$ inequivalent
% doubly even self-dual codes
% satisfying the condition (\ref{eq:Tdec}).
%such that the supports of the codewords of weight $4$
%are a $T$-decomposition.
%containing a unique $T$-decomposition.
We verified that, by a mapping
$\cC \mapsto B(\cC),$ the $19$ inequivalent
extremal even self-dual additive $\FF_4$-codes
of length $10$ give the $19$ inequivalent doubly even self-dual codes
of length $40$
satisfying (\ref{eq:Tdec}).
%such that the supports of the codewords of weight $4$
%are a $T$-decomposition.
% containing a unique $T$-decomposition.
%M This was verified by the {\sc Magma} function {\tt IsIsomorphic}.

\begin{rem}
Generator matrices of the $19$ extremal even self-dual additive
$\FF_4$-codes of length $10$ are listed in~\cite[Appendix]{BG00}, but
the generator matrix for $QC\_10l$ is the same as that for $QC\_10b$.
Note that only $QC\_10l$ has automorphism group of order
$16$~\cite[Table 2]{BG00}.
We found a code $\cC_{10}$ with $\#\Aut(\cC_{10})=16$
by~\cite[Theorem 6]{DP06} from a graph on $10$ vertices,
where $\Aut(\cC_{10})$ was determined
using the method given in~\cite[p.~1373]{CRSS},
by calculating
the automorphism groups of some related binary codes of length $30$.
A generator matrix of  $\cC_{10}$ is listed in Figure~\ref{Fig:F4}.
\end{rem}

\begin{figure}[htb]
\centering
{\small
%{\footnotesize
\[
\left(\begin{array}{cccccccccc}
\omega&1&1&1&1&1&0&0&0&0\\
1&\omega&0&0&0&1&1&0&0&0\\
1&0&\omega&0&1&0&0&1&1&1\\
1&0&0&\omega&0&1&0&1&0&0\\
1&0&1&0&\omega&0&1&0&1&1\\
1&1&0&1&0&\omega&0&0&0&0\\
0&1&0&0&1&0&\omega&0&0&1\\
0&0&1&1&0&0&0&\omega&1&0\\
0&0&1&0&1&0&0&1&\omega&0\\
0&0&1&0&1&0&1&0&0&\omega
\end{array}\right)
\]
\caption{A generator matrix of $\cC_{10}$}
\label{Fig:F4}
}
\end{figure}

%%\begin{rem}
There is a unique extremal even self-dual additive $\FF_4$-code of
length $6$~\cite[Table 1]{Hohn}, and the code $F_{24}$
in~\cite[Table II]{PS75} is a unique doubly even self-dual code of
length $24$ satisfying the condition (\ref{eq:Tdec}). There are
three inequivalent extremal even self-dual additive $\FF_4$-codes
of length $8$~\cite[Table 1]{Hohn}. These codes are No.~19, 20, 21
in~\cite[Table 1]{Hohn}, and we denote them by $\cC_{19}$,
$\cC_{20}$, $\cC_{21}$, respectively.
%As described in Remark~\ref{rem:2432},
%C$67$, C$68$ and C$69$ in~\cite[Table A]{CPS}
%are only the doubly even self-dual codes of length $32$
%containing a unique $T$-decomposition.
We verified that
$B(\cC_{19})$,  $B(\cC_{20})$ and $B(\cC_{21})$
are equivalent to C$67$, C$68$ and C$69$ in~\cite[Table A]{CPS},
respectively, which are
the doubly even self-dual codes of length $32$
satisfying the condition (\ref{eq:Tdec})
%such that the supports of the codewords of weight $4$
%are a $T$-decomposition
%containing a unique $T$-decomposition
(see Remark~\ref{rem:2432}).
%%\end{rem}
%%
Hence, there is a one-to-one correspondence between the
equivalence classes of doubly even self-dual codes of length $4n$
satisfying (\ref{eq:Tdec}) and the equivalence classes of extremal
even self-dual additive $\FF_4$-codes of length $n$ for
$n=6,8,10$.

% Hence, we have the following:
%
% \begin{prop}
% For $n=6,8,10$, there is a one-to-one
% correspondence between equivalence classes of
% doubly even self-dual codes of length $4n$
% satisfying (\ref{eq:Tdec}) and
% equivalence classes of
% even self-dual additive $\FF_4$-codes of length $n$ and minimum weight $4$.
% \end{prop}
%%%%%%%%%%%%% even self-dual additive $\FF_4$-codes

%%%
\subsection{Weight enumerators
$W_{40,C,\beta}$ ($\beta=1,2,\ldots,8$)}
In this subsection,
we continue a classification of extremal singly even self-dual
codes by modifying the classification
method for the case $\beta=10$,
which was given in the previous subsection.

%We suppose that $\beta=1,2,\ldots,8$.
%Let $C$ be an extremal singly even self-dual code
%with weight enumerator $W_{40,C,\beta}$. %%$(\beta=1,2,\ldots,8)$.
%A similar argument is useful for the classification of
%extremal singly even self-dual codes with weight enumerator
%$W_{40,C,\beta}$ where $\beta=7,8$.
Suppose that $1\leq\beta\leq 8$. 
Let $C$ be a doubly even self-dual code of length $40$.
By modifying the definition of $T$-decompositions, we consider a
collection $\{T_1,T_2, \ldots, T_\beta\}$ satisfying the
conditions (\ref{eq:T2}), (\ref{eq:T3}) and
\begin{equation}\label{eq:T4}
T_i \cap T_j =\emptyset
\end{equation}
for $i \ne j$. Denote by $N^\beta_{DE}(40)$ the number of the
inequivalent doubly even self-dual codes with $A_4=\beta$ and by
$\overline{N}^\beta_{DE}(40)$ the number of those codes among them
which satisfy (\ref{eq:T2}), (\ref{eq:T3})  and (\ref{eq:T4}). It
is trivial that $N^\beta_{DE}(40)=\overline{N}^\beta_{DE}(40)$ for
$\beta=1,2$. The numbers $N^\beta_{DE}(40)$ and
$\overline{N}^\beta_{DE}(40)$ for $\beta=1,2,\ldots,8$ are
obtained in~\cite{BHM40}, and the results are listed in
Table~\ref{Tab:2}.

%% %%%%%%%%%%%%%%%%%%%%%%%%%%%%%%%%%%%%%%%%%%%%%%%%
%% \begin{table}[thb]
%% \caption{Numbers $N_{d,\beta}$, $N'_{d,\beta}$ and $N_{s,\beta}$}
%% \label{Tab:2}
%% \begin{center}
%% {\small
%% %{\footnotesize
%% %{\scriptsize
%% \begin{tabular}{c|cccccccc}
%% \noalign{\hrule height0.8pt}
%% %\multicolumn{4}{c}{$(\#\Aut(C),R(C),\max N(C),\min N(C),\#N(C))$}\\
%% $\beta$ & 1&2&3&4&5&6&7&8 \\
%% \hline
%% $N_{d,\beta}$ & 20034&17276&12168& 8471& 5552& 3916& 2610& 1932 \\
%% $N'_{d,\beta}$ & 20034&17276&11241& 6645& 3115& 1380&  405&  120 \\
%% \hline
%% $N_{s,\beta}$&$4704206$& $1511827$& $337565$&
%% $64692$& $11009$& $2413$ & $405$ & $120$ \\
%% \noalign{\hrule height0.8pt}
%%   \end{tabular}
%% }
%% \end{center}
%% \end{table}
%% %%%%%%%%%%%%%%%%%%%%%%%%%%%%%%%%%%%%%%%%%%%%%%%%

%%%%%%%%%%%%%%%%%%%%%%%%%%%%%%%%%%%%%%%%%%%%%%%%
\begin{table}[thb]
\caption{$N^\beta_{DE}(40)$,
$\overline{N}^\beta_{DE}(40)$ and $N^\beta_{SE}(40)$
for $\beta=1,2,\ldots,8$}
\label{Tab:2}
\begin{center}
{\small
%{\footnotesize
%{\scriptsize
\begin{tabular}{c|rr|r}
\noalign{\hrule height0.8pt}
$\beta$ &
\multicolumn{1}{c}{$N^\beta_{DE}(40)$}&
\multicolumn{1}{c|}{$\overline{N}^\beta_{DE}(40)$}&
\multicolumn{1}{c}{$N^\beta_{SE}(40)$}\\
%% $N_{d,\beta}$ & $N'_{d,\beta}$ & $N_{s,\beta}$ \\
\hline
1 & 20034 & 20034 & 4674608 \\
2 & 17276 & 17276 & 1511827 \\
3 & 12168 & 11241 &  337565 \\
4 &  8471 &  6645 &   64692 \\
5 &  5552 &  3115 &   11009 \\
6 &  3916 &  1380 &    2413 \\
7 &  2610 &   405 &     405 \\
8 &  1932 &   120 &     120 \\
\noalign{\hrule height0.8pt}
  \end{tabular}
}
\end{center}
\end{table}
%%%%%%%%%%%%%%%%%%%%%%%%%%%%%%%%%%%%%%%%%%%%%%%%

%
For $\beta=1,2,\dots,8$,
a similar argument to that given in the previous subsection
shows that every
extremal singly even self-dual code with weight enumerator
$W_{40,C,\beta}$
can be constructed from some doubly even self-dual
code with $A_4=\beta$ containing $\{T_1,T_2, \ldots, T_\beta\}$,
which
satisfies the conditions
(\ref{eq:T2}), (\ref{eq:T3}), (\ref{eq:T4}).
Moreover, the argument
for equivalence tests given in the previous subsection
is also applied to the cases $\beta=1,2,\dots,8$.
Hence, for these cases,
we were able to complete
the classification of extremal singly even self-dual
codes with weight enumerator $W_{40,C,\beta}$.
%%%
For the case $\beta=1$,
the calculations took about 3 months
using 6 cores of a PC Intel i7 6 core processor.
The numbers $N^\beta_{SE}(40)$ of inequivalent extremal
singly even self-dual codes of length $40$ with weight enumerators
$W_{40,C,\beta}$ are also listed in Table~\ref{Tab:2}
for $\beta=1,2,\ldots,8$.
%%
%% \begin{prop}
%% There are
%% $4704206$, $1511827$, $337565$,
%% $64692$, $11009$, $2413$, $405$ and $120$
%% %> 4704206 + 1511827 + 337565 +  64692 + 11009 + 2413 + 405 + 120 +19;
%% %6632256
%% inequivalent extremal
%% singly even self-dual codes of length $40$ with weight enumerators
%% $W_{40,C,\beta}$, where $\beta=1,2,3,4,5,6,7$ and $8$, respectively.
%% \end{prop}
%%
%Combined with Proposition~\ref{prop:b10},
Combined with Proposition~\ref{prop:b10},
we have the following:

\begin{prop}
There are $6602658$
inequivalent  extremal singly even self-dual codes
of length $40$ with shadows of minimum weight $4$.
\end{prop}

In Table~\ref{Tab:AutCR1},
we list the numbers $N(\#\Aut,R)$ of inequivalent extremal
singly even self-dual codes
with automorphism groups of order $\#\Aut$
and covering radii $R$
for $\beta=1,2,\ldots,8$.

%%%%%%%%%%%%%%%%%%%%%%%%%%%%%%%%%%%%%%%%%%%%%%%%
\begin{table}[thbp]
\caption{$(\#\Aut,R,N(\#\Aut,R))$ for $\beta=1,2,\ldots,8$}
\label{Tab:AutCR1}
\begin{center}
%{\small
{\footnotesize
%{\scriptsize
\begin{tabular}{cccccc}
\noalign{\hrule height0.8pt}
\multicolumn{6}{c}{$(\#\Aut,R,N(\#\Aut,R))$}\\
\hline
\multicolumn{6}{c}{$\beta=1$}\\
\hline
$(1, 7, 370397)$ &
$(1, 8, 4235394)$ &
$(2, 7, 8663)$ &
$(2, 8, 53572)$ &
$(3, 7, 38)$ &
$(3, 8, 302)$ \\
$(4, 7, 908)$ &
$(4, 8, 3965)$ &
$(6, 7, 21)$ &
$(6, 8, 75)$ &
$(8, 7, 153)$ &
$(8, 8, 645)$ \\
$(12, 7, 4)$ &
$(12, 8, 51)$ &
$(14, 7, 1)$ &
$(16, 7, 98)$ &
$(16, 8, 160)$ &
$(18, 7, 1)$ \\
$(18, 8, 2)$ &
$(21, 8, 1)$ &
$(24, 7, 5)$ &
$(24, 8, 25)$ &
$(32, 7, 47)$ &
$(32, 8, 45)$ \\
$(36, 8, 4)$ &
$(48, 7, 4)$ &
$(48, 8, 4)$ &
$(64, 7, 3)$ &
$(64, 8, 3)$ &
$(72, 8, 4)$ \\
$(96, 7, 2)$ &
$(96, 8, 3)$ &
$(128, 7, 3)$ &
$(192, 8, 2)$ &
$(1008, 8, 1)$ &
$(1728, 7, 1)$ \\
$(4032, 8, 1)$ \\
% \noalign{\hrule height0.8pt}
%   \end{tabular}
% }
% \end{center}
% \end{table}
% %%%%%%%%%%%%%%%%%%%%%%%%%%%%%%%%%%%%%%%%%%%%%%%%
% %%%%%%%%%%%%%%%%%%%%%%%%%%%%%%%%%%%%%%%%%%%%%%%%
% \begin{table}[thbp]
% \caption{$(\#\Aut,R,N(\#\Aut,R))$ for $\beta=2$}
% \label{Tab:AutCR2}
% \begin{center}
% %{\small
% {\footnotesize
% %{\scriptsize
% \begin{tabular}{cccccc}
% \noalign{\hrule height0.8pt}
% \multicolumn{6}{c}{$(\#\Aut,R,N(\#\Aut,R))$}\\
\hline
\multicolumn{6}{c}{$\beta=2$}\\
\hline
$(1, 7, 96666)$ &
$(1, 8, 1359341)$ &
$(2, 7, 5464)$ &
$(2, 8, 39474)$ &
$(3, 7, 10)$ &
$(3, 8, 117)$ \\
$(4, 7, 932)$ &
$(4, 8, 6042)$ &
$(6, 7, 5)$ &
$(6, 8, 76)$ &
$(8, 7, 167)$ &
$(8, 8, 1760)$ \\
$(12, 7, 10)$ &
$(12, 8, 34)$ &
$(16, 7, 108)$ &
$(16, 8, 741)$ &
$(24, 7, 2)$ &
$(24, 8, 16)$ \\
$(32, 7, 51)$ &
$(32, 8, 322)$ &
$(48, 8, 15)$ &
$(64, 7, 9)$ &
$(64, 8, 208)$ &
$(96, 7, 1)$ \\
$(96, 8, 7)$ &
$(128, 7, 1)$ &
$(128, 8, 106)$ &
$(144, 8, 1)$ &
$(192, 8, 6)$ &
$(256, 7, 3)$ \\
$(256, 8, 52)$ &
$(288, 8, 1)$ &
$(384, 8, 4)$ &
$(512, 8, 24)$ &
$(768, 7, 1)$ &
$(768, 8, 7)$ \\
$(1024, 7, 1)$ &
$(1024, 8, 12)$ &
$(1536, 8, 2)$ &
$(2048, 7, 1)$ &
$(2048, 8, 12)$ &
$(4096, 8, 5)$ \\
$(4608, 8, 1)$ &
$(6144, 8, 1)$ &
$(8192, 8, 1)$ &
$(12288, 8, 2)$ &
$(16384, 8, 2)$ &
$(98304, 7, 1)$ \\
$(196608, 8, 2)$ & \\
% \noalign{\hrule height0.8pt}
%   \end{tabular}
% }
% \end{center}
% \end{table}
% %%%%%%%%%%%%%%%%%%%%%%%%%%%%%%%%%%%%%%%%%%%%%%%%
% %%%%%%%%%%%%%%%%%%%%%%%%%%%%%%%%%%%%%%%%%%%%%%%%
% \begin{table}[thbp]
% \caption{$(\#\Aut,R,N(\#\Aut,R))$ for $\beta=3$}
% \label{Tab:AutCR3}
% \begin{center}
% %{\small
% {\footnotesize
% %{\scriptsize
% \begin{tabular}{cccccc}
% \noalign{\hrule height0.8pt}
% \multicolumn{6}{c}{$(\#\Aut,R,N(\#\Aut,R))$}\\
\hline
\multicolumn{6}{c}{$\beta=3$}\\
\hline
$(1,7,13222) $&
$(1,8,300610)$&
$(2, 7,2014) $&
$(2,8,16954) $&
$(3, 7, 8)   $&
$(3, 8, 46)  $\\
$(4, 7, 482) $&
$(4, 8,2923) $&
$(6, 7, 8)   $&
$(6, 8, 51)  $&
$(8, 7, 123) $&
$(8, 8, 643) $\\
$(12, 7, 6)  $&
$(12, 8, 48) $&
$(16, 7, 56) $&
$(16, 8,214) $&
$(24, 7, 4)  $&
$(24, 8, 11) $\\
$(32, 7, 23) $&
$(32, 8, 84) $&
$(48, 7, 1)  $&
$(48, 8, 12) $&
$(64, 7, 3)  $&
$(64, 8, 13) $\\
$(96, 7, 1)  $&
$(128, 8, 2) $&
$(144, 7, 1) $&
$(192, 7, 1) $&
$(384, 8, 1) $&\\
% \noalign{\hrule height0.8pt}
%   \end{tabular}
% }
% \end{center}
% \end{table}
% %%%%%%%%%%%%%%%%%%%%%%%%%%%%%%%%%%%%%%%%%%%%%%%%
% %%%%%%%%%%%%%%%%%%%%%%%%%%%%%%%%%%%%%%%%%%%%%%%%
% \begin{table}[thbp]
% \caption{$(\#\Aut,R,N(\#\Aut,R))$ for $\beta=4$}
% \label{Tab:AutCR4}
% \begin{center}
% %{\small
% {\footnotesize
% %{\scriptsize
% \begin{tabular}{cccccc}
% \noalign{\hrule height0.8pt}
% \multicolumn{6}{c}{$(\#\Aut,R,N(\#\Aut,R))$}\\
\hline
\multicolumn{6}{c}{$\beta=4$}\\
\hline
$(1, 7, 708)$ &
$(1, 8, 51223)$ &
$(2, 7, 349)$ &
$(2, 8, 7605)$ &
$(3, 7, 4)$ &
$(3, 8, 42)$ \\
$(4, 7, 132)$ &
$(4, 8, 2309)$ &
$(6, 7, 6)$ &
$(6, 8, 32)$ &
$(8, 7, 58)$ &
$(8, 8, 943)$ \\
$(12, 7, 5)$ &
$(12, 8, 29)$ &
$(16, 7, 24)$ &
$(16, 8, 454)$ &
$(18, 7, 1)$ &
$(18, 8, 2)$ \\
$(24, 7, 3)$ &
$(24, 8, 18)$ &
$(32, 7, 14)$ &
$(32, 8, 284)$ &
$(36, 8, 2)$ &
$(48, 8, 7)$ \\
$(64, 7, 1)$ &
$(64, 8, 162)$ &
$(72, 8, 2)$ &
$(96, 7, 4)$ &
$(96, 8, 7)$ &
$(128, 7, 2)$ \\
$(128, 8, 108)$ &
$(144, 8, 1)$ &
$(192, 7, 1)$ &
$(192, 8, 4)$ &
$(256, 7, 2)$ &
$(256, 8, 36)$ \\
$(288, 7, 1)$ &
$(288, 8, 1)$ &
$(384, 8, 9)$ &
$(512, 7, 1)$ &
$(512, 8, 29)$ &
$(576, 8, 2)$ \\
$(768, 8, 4)$ &
$(1024, 8, 12)$ &
$(1536, 8, 11)$ &
$(1728, 8, 1)$ &
$(2048, 8, 14)$ &
$(3072, 8, 3)$ \\
$(4096, 8, 3)$ &
$(6144, 8, 3)$ &
$(8192, 8, 6)$ &
$(9216, 8, 1)$ &
$(16384, 8, 1)$ &
$(20736, 7, 1)$ \\
$(49152, 7, 1)$ &
$(98304, 8, 2)$ &
$(294912, 8, 2)$\\
%%%%%%%%%%%%%%%%%%%%%%%%%%%%%%%%%%%%%
\noalign{\hrule height0.8pt}
  \end{tabular}
}
\end{center}
\end{table}
%%%%%%%%%%%%%%%%%%%%%%%%%%%%%%%%%%%%%%%%%%%%%%%%

\setcounter{table}{2}
%%%%%%%%%%%%%%%%%%%%%%%%%%%%%%%%%%%%%%%%%%%%%%%%
\begin{table}[thb]
\caption{$(\#\Aut,R,N(\#\Aut,R))$ for $\beta=1,2,\ldots,8$ (continued)}
%\label{Tab:AutCR5}
\begin{center}
%{\small
{\footnotesize
%{\scriptsize
\begin{tabular}{cccccc}
\noalign{\hrule height0.8pt}
\multicolumn{6}{c}{$(\#\Aut,R,N(\#\Aut,R))$}\\
\hline
\multicolumn{6}{c}{$\beta=5$}\\
\hline
$(1, 7, 51)$  &
$(1, 8, 6723)$&
$(2, 7, 75)$  &
$(2, 8, 2498)$&
$(3, 8, 14)$  &
$(4, 7, 47)$  \\
$(4, 8, 974)$ &
$(6, 7, 1)$   &
$(6, 8, 12)$  &
$(8, 7, 20)$  &
$(8, 8, 348)$ &
$(10, 8, 5)$  \\
$(12, 7, 5)$  &
$(12, 8, 19)$ &
$(16, 7, 11)$ &
$(16, 8, 119)$&
$(24, 8, 8)$  &
$(32, 7, 8)$  \\
$(32, 8, 42)$ &
$(48, 8, 5)$  &
$(64, 7, 2)$  &
$(64, 8, 11)$ &
$(96, 8, 1)$  &
$(120, 8, 2)$ \\
$(128, 7, 1)$ &
$(128, 8, 3)$ &
$(144, 8, 1)$ &
$(192, 8, 1)$ &
$(320, 7, 1)$ &
$(384, 8, 1)$ \\
% \noalign{\hrule height0.8pt}
%   \end{tabular}
% }
% \end{center}
% \end{table}
% %%%%%%%%%%%%%%%%%%%%%%%%%%%%%%%%%%%%%%%%%%%%%%%%
% %%%%%%%%%%%%%%%%%%%%%%%%%%%%%%%%%%%%%%%%%%%%%%%%
% \begin{table}[thbp]
% \caption{$(\#\Aut,R,N(\#\Aut,R))$ for $\beta=6$}
% \label{Tab:AutCR6}
% \begin{center}
% %{\small
% {\footnotesize
% %{\scriptsize
% \begin{tabular}{cccccc}
% \noalign{\hrule height0.8pt}
% \multicolumn{6}{c}{$(\#\Aut,R,N(\#\Aut,R))$}\\
% \hline
\hline
\multicolumn{6}{c}{$\beta=6$}\\
\hline
$(1, 8, 430)$ &
$(2, 7, 9)$   &
$(2, 8, 676)$ &
$(3, 8, 1)$   &
$(4, 7, 8)$   &
$(4, 8, 424)$ \\
$(6, 8, 5)$   &
$(8, 7, 8)$   &
$(8, 8, 268)$ &
$(12, 7, 1)$  &
$(12, 8, 9)$  &
$(16, 7, 7)$  \\
$(16, 8, 173)$&
$(24, 8, 6)$  &
$(32, 7, 3)$  &
$(32, 8, 110)$&
$(48, 7, 1)$  &
$(48, 8, 6)$  \\
$(64, 7, 2)$  &
$(64, 8, 84)$ &
$(72, 8, 1)$  &
$(96, 7, 1)$  &
$(96, 8, 4)$  &
$(128, 7, 2)$ \\
$(128, 8, 70)$&
$(144, 7, 1)$ &
$(192, 7, 1)$ &
$(256, 8, 22)$&
$(288, 7, 1)$ &
$(288, 8, 1)$ \\
$(384, 8, 4)$ &
$(512, 7, 1)$ &
$(512, 8, 29)$&
$(768, 8, 6)$ &
$(1024, 8, 8)$&
$(1536, 8, 2)$\\
$(2048, 8, 9)$&
$(3072, 8, 4)$&
$(6144, 7, 1)$&
$(6144, 8, 5)$&
$(8192, 8, 5)$&
$(24576, 8, 3)$\\
$(49152, 8, 1)$ \\
% \noalign{\hrule height0.8pt}
%   \end{tabular}
% }
% \end{center}
% \end{table}
% %%%%%%%%%%%%%%%%%%%%%%%%%%%%%%%%%%%%%%%%%%%%%%%%
% %%%%%%%%%%%%%%%%%%%%%%%%%%%%%%%%%%%%%%%%%%%%%%%%
% \begin{table}[thbp]
% \caption{$(\#\Aut,R,N(\#\Aut,R))$ for $\beta=7$}
% \label{Tab:AutCR7}
% \begin{center}
% %{\small
% {\footnotesize
% %{\scriptsize
% \begin{tabular}{cccccc}
% \noalign{\hrule height0.8pt}
% \multicolumn{6}{c}{$(\#\Aut,R,N(\#\Aut,R))$}\\
\hline
\multicolumn{6}{c}{$\beta=7$}\\
\hline
$( 4, 8, 109)$ &
$( 8, 7,   1)$ &
$( 8, 8, 138)$ &
$(12, 8,   4)$ &
$(16, 7,   2)$ &
$(16, 8,  79)$ \\
$(24, 7,   1)$ &
$(24, 8,   6)$ &
$(32, 7,   1)$ &
$(32, 8,  19)$ &
$(48, 8,   8)$ &
$(64, 7,   1)$ \\
$(64, 8,  13)$ &
$(72, 8,   2)$ &
$(96, 7,   1)$ &
$(96, 8,   8)$ &
$(128, 8,  1)$ &
$(144, 8,  1)$ \\
$(192, 7,  1)$ &
$(192, 8,  1)$ &
$(288, 7,  1)$ &
$(384, 7,  1)$ &
$(384, 8,  2)$ &
$(576, 8,  2)$ \\
$(1728, 7, 1)$ &
$(5184, 8, 1)$ \\
% \noalign{\hrule height0.8pt}
%   \end{tabular}
% }
% \end{center}
% \end{table}
% %%%%%%%%%%%%%%%%%%%%%%%%%%%%%%%%%%%%%%%%%%%%%%%%
% %%%%%%%%%%%%%%%%%%%%%%%%%%%%%%%%%%%%%%%%%%%%%%%%
% \begin{table}[thbp]
% \caption{$(\#\Aut,R,N(\#\Aut,R))$ for $\beta=8$}
% \label{Tab:AutCR8}
% \begin{center}
% %{\small
% {\footnotesize
% %{\scriptsize
% \begin{tabular}{cccccc}
% \noalign{\hrule height0.8pt}
% \multicolumn{6}{c}{$(\#\Aut,R,N(\#\Aut,R))$}\\
% \hline
\hline
\multicolumn{6}{c}{$\beta=8$}\\
\hline
$( 32, 8,  6)$ &
$( 64, 8, 21)$ &
$( 96, 8,  1)$ &
$(128, 8, 26)$ &
$(192, 8,  3)$ &
$(256, 8, 11)$ \\
$(384, 8,  4)$ &
$(512, 8, 17)$ &
$(768, 7,  1)$ &
$(1024,7, 1)$ &
$(1024,8, 4)$ &
$(1536,8, 5)$ \\
$(2048,8, 5)$ &
$(3072,8, 4)$ &
$(4096,8, 2)$ &
$(5376,8, 1)$ &
$(8192,8, 3)$ &
$(12288, 8, 1)$ \\
$(16384, 8, 2)$ &
$(36864, 8, 1)$ &
$(688128, 7, 1)$ \\
\noalign{\hrule height0.8pt}
  \end{tabular}
}
\end{center}
\end{table}
%%%%%%%%%%%%%%%%%%%%%%%%%%%%%%%%%%%%%%%%%%%%%%%%

%%%%%
\subsection{Another approach and weight enumerator $W_{40,C,0}$}
\label{Subsection:0}
The classification of extremal self-dual codes of length $40$ was
also completed by an approach which does not depend on the weight
enumerators. As a consequence, we have a classification of
extremal singly even self-dual codes with weight enumerator
$W_{40,C,0}$, which is the remaining case. Here we describe how
such a classification was done. %{\bf (I emphasized that your
%method include all extremal cases!)}

%{\bf For the next paragragh, I have the following comments:
%\begin{itemize}
%\item
%I think it is better to concentrate your method, instead
%of listing other methods, because we can write
%how your algorithm is better (and readers feel this from
%the result).

%\item
%I think ``Harada-Munemasa construction'' is not suitable
%as the name, because (as we (H--M) described) the method
%is similar to that by Huffman where he considers $\FF_4$-codes
%but its application to binary is very trivial.
%\end{itemize}
%}
From the self-dual codes of length $38$ and minimum weights $6$
and $8$, we complete a classification of the extremal self-dual
codes of length $40$, using the algorithm presented in
\cite{BB38}, noting that a classification of self-dual codes
of length $38$ was done in \cite{BB38}. We emphasize that the used
algorithm has a better way to deal with equivalence classes and it
gives as output exactly one representative of every equivalence
class. The constructive part of the algorithm is not different
from the other recursive constructions of self-dual codes,
which were listed in \cite{AGKSS38},
but to take only one representative of any
equivalence class, we used a completely different manner. Its
special feature is that practically there is not equivalence test
for the objects. Algorithms of this type are known as
isomorph-free generation \cite{McK}.

%%%%%%%%%%%%%%%%%%%%%%%%%%%%%%%%%%%%%%%%%%%%%%%%%%%%%%%%%%%%%%%%%%
The algorithm for isomorph-free generation solves two main
problems. The first one is to find only inequivalent objects (in
our case extremal self-dual codes of length 40) using a basic
object $B$ from the previous step (in our case a self-dual code of
length $38$ and minimum weight $6$ or $8$). It is easy to solve
this by defining an action of the automorphism group $\Aut(B)$ on
the set of the objects constructed from $B$. For this
construction, we use the method described in \cite{HM36}.

The second main problem is how to take only the inequivalent
objects among already constructed ones (after solving the first
problem). This problem and our solution is explained in details in
\cite{BB38}. We give here only the general idea. We use the
concept for a canonical labeling map and a canonical
representative of an equivalence class. We consider here the
action of the symmetric group $S_n$ on the set of all self-dual
codes of length $n$. We fix a so called \emph{canonical
representative} for any equivalence class which is selected on the
base of some specific conditions. The canonical representative is
intended to make easily a distinction among the equivalence
classes. The function which maps any code $C$ to the canonical
representative of its equivalence class is called a
\emph{canonical representative map}. This map is realized by a
proper algorithm which uses invariants. Using the canonical
representative, we define an ordering of the coordinates of the
code $C$. This ordering is unique only if $\#\Aut(C)=1$. In
the other cases there are $\#\Aut(C)$ such orderings. To take
only the inequivalent codes we use a parent test for each code
obtained in the first part of the algorithm. An example for such
test is the following: A code passes the test if the added two
coordinates during the construction are the biggest (last) two
coordinates in the defined ordering.
%%%%%%%%%%%%%%%%%%%%%%%%%%%%%%%%%%%%%%%%%%%%%%%%%%%%%%%%%%%%%%%%%%

Using the above algorithm, all inequivalent extremal doubly even
self-dual codes as well as extremal singly even self-dual codes of
length $40$ were obtained. The calculations took about two months
using four cores of a PC Intel i5 4 core processor. We have the
following:

\begin{thm}\label{thm:40}
There are $10200655$ inequivalent  extremal singly even self-dual
codes of length $40$.
\end{thm}

As a summary, we list in Table~\ref{Tab:40} the numbers
$N^\beta_{SE}(40)$ of inequivalent  extremal singly even self-dual
codes with weight enumerators $W_{40,C,\beta}$.

%%%%%%%%%%%%%%%%%%%%%%%%%%%%%%%%%%%%%%%%%%%%%%%%
\begin{table}[thb]
\caption{Extremal singly even self-dual codes of length $40$}
\label{Tab:40}
\begin{center}
{\small
%{\footnotesize
%{\scriptsize
\begin{tabular}{c|r||c|r}
\noalign{\hrule height0.8pt} $\beta$ &
\multicolumn{1}{c||}{$N^\beta_{SE}(40)$} &
$\beta$ & \multicolumn{1}{c}{$N^\beta_{SE}(40)$}\\
\hline
0 & 3597997 &5 &   11009 \\
1 & 4674608 &6 &    2413 \\
2 & 1511827 &7 &     405 \\
3 &  337565 &8 &     120 \\
4 &   64692 &10&      19 \\
\noalign{\hrule height0.8pt}
  \end{tabular}
}
\end{center}
\end{table}
%%%%%%%%%%%%%%%%%%%%%%%%%%%%%%%%%%%%%%%%%%%%%%%%

For the weight enumerator $W_{40,C,0}$, we list in
Table~\ref{Tab:AutCR0} the numbers $N(\#\Aut,R)$ of inequivalent
extremal singly even self-dual codes with automorphism groups of
order $\#\Aut$ and covering radii $R$. Since the shadow of $C$ is
a coset of $C$ having minimum weight $8$, we have that $R \ge 8$.
It follows from the classification that there is no extremal
singly even self-dual code with $R \ge 9$.

%%%%%%%%%%%%%%%%%%%%%%%%%%%%%%%%%%%%%%%%%%%%%%%%
\begin{table}[thbp]
\caption{$(\#\Aut,R,N(\#\Aut,R))$ for $\beta=0$}
\label{Tab:AutCR0}
\begin{center}
%{\small
{\footnotesize
%{\scriptsize
\begin{tabular}{cccccc}
\noalign{\hrule height0.8pt}
\multicolumn{6}{c}{$(\#\Aut,R,N(\#\Aut,R))$}\\
\hline $(1, 8, 3542831)$& $(2, 8, 48796)$& $(3, 8, 222)$& $(4, 8,
4273)$& $(5, 8, 4)$ &
$(6, 8, 133)$\\
$(8, 8, 904)$& $(10, 8, 5)$& $(12, 8, 59)$& $(16, 8, 379)$ & $(20,
8, 4)$&
$(24, 8, 29)$\\
$(32, 8, 113)$& $(40, 8, 2)$& $(48, 8, 19)$ & $(64, 8, 103)$&
$(96, 8, 3)$&
$(120, 8, 5)$\\
$(128, 8, 35)$& $(160, 8, 1)$ & $(192, 8, 6)$& $(256, 8, 22)$&
$(384, 8, 2)$&
$(512, 8, 14)$\\
$(640, 8, 1)$ & $(768, 8, 3)$& $(1024, 8, 6)$& $(1536, 8, 4)$&
$(2048, 8, 7)$&
$(3072, 8, 2)$ \\
$(4096, 8, 2)$& $(6144, 8, 1)$& $(8192, 8, 1)$& $(12288, 8, 1)$&
$(16384, 8, 2)$&
$(20480, 8, 1)$\\
$(24576, 8, 1)$& $(491520, 8, 1)$&
& \\
\noalign{\hrule height0.8pt}
  \end{tabular}
}
\end{center}
\end{table}
%%%%%%%%%%%%%%%%%%%%%%%%%%%%%%%%%%%%%%%%%%%%%%%%

%%%%%%%%%%%%%%%%%%%%%%%%%%%%%%%%%%%%%%%%%%%%%%%%%%%%%%%%%%%%%%%%%%%%%%%%%%%%%%%%%
\subsection{Automorphism groups}\label{Sect:Groups}
%{\bf (Here is minor changed)}

Here we consider automorphism groups of extremal
self-dual codes of length $40$.
% It was known that if a prime $p$ divides the order
% of the automorphism group of an extremal
% self-dual code of length $40$
% then $p \in \{2,3,5,7,19\}$ (see \cite{BY40}),
% before the classification was completed.
It can be shown without the classification
that if a prime $p$ divides the order
of the automorphism group of an extremal
self-dual code of length $40$
then $p \in \{2,3,5,7,19\}$ (see \cite{BY40}).
By \cite{BHM40} and Theorem \ref{thm:40},
there are $10 214 125$ extremal self-dual codes of length $40$.
Their automorphism groups are divided into $91$ different orders.
We present the number of the codes in respect to the order of
their automorphism group in Table \ref{Table:aut_groups}.

The $9 972 575$ codes
(more than $97\%$ of $10 214 125$ codes)
have a trivial automorphism group (group of order
1).
This means that only $228 849$ codes
have nontrivial automorphism groups.
There are $189 449$ codes with
automorphism group of order $2$.
So the codes with automorphism
group of order $1$ or $2$ are more than $99.61\%$ of all inequivalent
codes. This calculation substantiates some conjectures from
\cite{BB38} and very small number of the codes have bigger
automorphism groups.
%%%
% Instead of these $189 449$ codes, another $36979$ codes have automorphism
% group of order $2^s$ for $s\ge 2$.
%One code has an  automorphismgroup of
%order $2^{18}=262144$ - it is a singly even self-dual
%code with weight enumerator $W_{40,C,10}$.
% (a reader knows this from the table)

% {\bf (Maybe here can be removed?)}
% There are three codes with very rich automorphism groups. One
% doubly even code which we denote by $C_{r0}$ has an
% automorphism group of order
% $82 575 360 = 2^{18}\cdot 3^2\cdot 5\cdot 7$. Other two codes, one
% doubly even - $C_{r1}$, and one singly even $C_{r2}$ with weight
% enumerator $W_{40,C,10}$, have automorphism groups of order
% $44 236 800 =2^{16}\cdot 3^3\cdot 5^2$. $C_{r1}$ and $C_{r2}$ are the only
% extremal self-dual codes of length $40$
% having automorphisms of type $5$-$(4,20)$ \cite{BY40},
% \cite{Yorus}.
%
% {\bf (Maybe here can be removed?)}
% According to \cite{Yorus},
% there are three, five and $45$ inequivalent extremal doubly even
% self-dual codes of length $40$ having an automorphism of orders
% $19,7$ and $5$, respectively.
% This result was approved in \cite{BHM40}.
% % One of these codes has an automorphism group of order $6840$ which
% % is divisible by $5\cdot 19$, and another one, namely $C_{r0}$,
% % has a group divisible by $35$.

Extremal singly even self-dual codes of length $40$ having
automorphisms of odd prime order $p\ge 5$ were considered
in \cite{BY40}.
Here we give the current knowledge on the existence of
such codes for $p \ge 3$.

\begin{itemize}
\item $p=19$:
No such code exists for this case.
\item $p=7$:
%Bouyuklieva and Yorgov \cite{BY40} constructed two codes with an
%automorphism of order $7$.
Unfortunately, there is a mistake in \cite{BY40} because there are
six inequivalent extremal singly even self-dual codes of length $40$ with
an automorphism of order 7 - two codes with weight enumerator
$W_{40,C,8}$ and automorphism groups of orders $688 128
=2^{15}\cdot 21$ and $5376 = 2^{8}\cdot 21$, and four codes with
weight enumerator $W_{40,C,1}$ and automorphism groups of orders
$4032 = 2^{6}\cdot 63$, $1008 = 2^{4}\cdot 63$, $21$ and $14$.
Bouyuklieva and Yorgov \cite{BY40} missed four of the codes
because they had to consider five different supports for the
codeword of weight 2 in $\pi(F_{\sigma}(C))$, namely $\{1,6\}$,
$\{2,6\}$, $\{3,6\}$, $\{4,6\}$ and $\{5,6\}$. So the codes with
weight enumerator $W_{40,C,8}$
could be constructed as a combination of the matrix
$H_4$ and $\pi(F_{\sigma}(C))$ with $\{2,6\}$ as a support of the
weight 2 vector, and $H_3$ with $\{4,6\}$ as this support. The
additional two codes with weight enumerator $W_{40,C,1}$
are obtained from $H_1$ with
$\{2,6\}$ and $H_2$ with $\{3,6\}$ (see \cite{BY40} for the
details of the construction and for the matrices $H_1$, $H_2$,
$H_3$, $H_4$).
\item $p=5$:
Exactly $39$ inequivalent singly even self-dual
codes of length $40$ have an automorphism of order 5 - two codes more
than it is stated in \cite{BY40}. The authors missed there two
codes with weight enumerator $W_{40,C,0}$.
% {\bf (What is the reason?  (like the case $p=7$)}

\item $p=3$:
% For the smallest odd prime case $p=3$,
% we have the following remark.
% \begin{rem}
% Exactly $2373$ inequivalent extremal self-dual codes
% of length $40$ have an
% automorphism of order $3$, $1986$ of which are singly even.
Exactly $1986$ inequivalent extremal singly even self-dual codes
of length $40$ have an automorphism of order $3$.
% \end{rem}

\end{itemize}

We formulate these corrections on the cases $p=5$ and $7$.

\begin{cor}\label{aut5-7}
There are exactly $6$ and $39$
inequivalent extremal singly even self-dual
codes of length $40$ having an automorphism of order $7$ and $5$,
respectively.
\end{cor}

%% I also checked this table (H)
\begin{table}[htbp]\rm
\begin{center}
\caption{Automorphism groups of the extremal self-dual codes of
length 40}\vspace*{0.2in} \label{Table:aut_groups} {\footnotesize
\begin{tabular}{c|c|c||c|c|c||c|c|c}
\noalign{\hrule height1pt} $\#\Aut(C)$&singly&doubly&$\#\Aut(C)$&singly&doubly&$\#\Aut(C)$&singly&doubly\\
&even&even&&even&even&&even&even\\
\hline \multicolumn{9}{c}{$\#\Aut(C)=2^s$, $s=0,1,\dots,16,18$}\\
\hline  1&9977596&10400&64&639&75&4096&12&1\\
\hline 2&186149&3538&128&360&46&8192&16&--\\
\hline 4&23528&1189&256&148&21&16384&8&1\\
\hline 8&6179&459&512&115&16&32768&1&1\\
\hline 16&2625&233&1024&44&3&65536&1&1\\
\hline 32&1172&70&2048&48&4&262144&1&--\\
 \hline
 \multicolumn{9}{c}{$\#\Aut(C)=3\cdot 2^s$, $s=0,1,\dots,16,18$}\\
\hline  3&804&43&192&27&12&12288&5&2\\
\hline 6&425&68&384&28&12&24576&4&--\\
\hline 12&284&80&768&22&7&49152&5&3\\
\hline 24&134&41&1536&24&10&98304&3&--\\
\hline 48&82&34&3072&13&3&196608&2&--\\
\hline 96&44&12&6144&11&7&786432&1&1\\
\hline
 \multicolumn{9}{c}{$\#\Aut(C)=2^s\cdot 3^m$, $s=1,\dots,17$, $m=2,3,4$}\\
\hline  18&6&1&1296&--&1&20736&1&1\\
\hline 36&6&1&1728&3&1&36864&1&--\\
\hline 72&9&4&4608&1&2&110592&--&1\\
\hline 144&6&4&5184&1&--&147456&1&1\\
\hline 288&6&4&9216&1&1&294912&2&--\\
\hline 576&4&3&18432&1&1&1179648&1&--\\
\hline
 \multicolumn{9}{c}{$\#\Aut(C)=2^s\cdot 3^m\cdot 5$, $s=0,1,\dots,18$, $m=0,1,2$}\\
\hline  5&4&2&240&--&2&245760&1&1\\
\hline  10&10&8&320&1&1&327680&1&--\\
\hline  20&4&4&640&1&--&491520&1&--\\
\hline  30&--&2&720&--&2&737280&1&1\\
\hline  40&2&5&1920&--&1&983040&--&1\\
\hline  60&--&2&3840&--&1&1474560&1&1\\
\hline  120&7&5&20480&2&1&11796480&1&--\\
\hline  160&1&1&61440&--&1&&&\\
\hline
 \multicolumn{9}{c}{$\#\Aut(C)=2^s\cdot 3^m\cdot 7$, $s=0,1,\dots,18$, $m=0,1,2$}\\
\hline  14&1&--&2688&--&1&688128&1&--\\
\hline  21&1&--&4032&1&--&5505024&--&1\\
\hline  1008&1&--&5376&1&1&8257536&--&1\\
\hline
 \multicolumn{9}{c}{$\#\Aut(C)=19q$, $q=2$, 6, 360}\\
\hline  38&--&1&114&--&1&6840&--&1\\
\hline
 \multicolumn{9}{c}{$\#\Aut(C)>$ 40 000 000}\\
\hline  &&&$44 236 800$&1&1&$82 575 360$&--&1\\
\noalign{\hrule height1pt}
\end{tabular}}
\end{center}
\end{table}

%%%%%
\subsection{Some observations}

By the sphere-covering bound (see~\cite{A-P}),
the covering radius of a self-dual code of length $40$
is at least $6$.
From Tables~\ref{Tab:AutCR1} and~\ref{Tab:AutCR0},
there are $501337$ inequivalent
% I corrected the number of codes with CR=7 and the
% old number is based on my incorrect classification
% (I am sorry, H)
extremal singly even self-dual codes
of length $40$ with covering radius $7$, and there is no
extremal singly even self-dual code
with covering radius $6$.
Note that
there are only two inequivalent extremal doubly even self-dual codes
with covering radius $7$, and there is no
extremal doubly even self-dual code
with covering radius $6$~\cite{BHM40}.
%% upper bound on CR
By the Delsarte bound (see~\cite{A-P}),
the covering radius of an extremal doubly even self-dual
code of length $40$ is at most $8$.
It follows from the classification that
the covering radius of an extremal singly even self-dual code
of the same length is also at most $8$.

Let $N_{DE}(n)$ and $N_{SE}(n)$ be the numbers of inequivalent
extremal doubly even self-dual codes and singly even self-dual
codes of length $n$ and minimum weight $4\lfloor n/24 \rfloor +4$,
respectively.
Then it holds that
\begin{align*}
(N_{DE}(8),N_{SE}(8))  &=(1,0), \\
(N_{DE}(16),N_{SE}(16))&=(2,1), \\
(N_{DE}(24),N_{SE}(24))&=(1,0), \\
(N_{DE}(32),N_{SE}(32))&=(5,3), \\
(N_{DE}(40),N_{SE}(40))&=(16470, 10200655).
%& (N_d(48),N_s(48))=(1,0).
\end{align*}
%There are many more extremal singly even self-dual codes
%of length $40$ than extremal doubly even self-dual codes
%of that length,
It follows that
$40$ is the smallest length $n$ with
$N_{DE}(n) < N_{SE}(n)$.

%%%%%%%%%%%%%%%%%%%%%%%%%%%%%%%%%%%%%%%%
\section{Odd unimodular lattices in dimension 40}
\label{Sec:L}

For odd unimodular lattice, we consider a situation which is
similar to that for singly even self-dual codes given
in Section~\ref{Subsec:C}.
We show that the number of vectors of norm $2$
in the shadow of an extremal odd unimodular lattice in dimension
$40$ is at most $80$.
We also give a classification of extremal odd unimodular lattices in dimension
$40$ with shadows having $80$ vectors of norm $2$.

%%%%%%%%%%%%%%
\subsection{Frames of Type A, B and C}

Let $L$ be an even integral lattice in dimension $n$.
Let $e_1,e_2,\ldots,e_n$ be vectors of $\RR^n$
satisfying
\begin{equation}\label{eq:C1}
(e_i,e_j)=2\delta_{ij} \text{ and }
e_i\pm e_j \in L \ (1 \le i,j \le n),
\end{equation}
where $\delta_{ij}$ is the Kronecker delta. If the vectors
$e_1,e_2,\ldots,e_n$ satisfy the following conditions
\begin{align}
\label{eq:CA}
&e_1,e_2,\ldots,e_n \in L,\\
\label{eq:CB}
&e_1,e_2,\ldots,e_n
\not\in L \text{ but }\frac{1}{2} \sum_{i=1}^n \ZZ e_i \supset L,\\
\label{eq:CC}
&\frac{1}{2} \sum_{i=1}^n \ZZ e_i \not\supset L,
\end{align}
%% In this case, we say that  $L$ has a frame of Type C (see~\cite{KKM}).
then the set of these vectors is called
{\em a frame} of Type A, B and C related to $L$,
respectively~\cite{KKM}.

Let $D$ be a doubly even code of length $n \equiv 0 \pmod 8$.
Let $e_1,e_2,\ldots,e_n$ be vectors of $\RR^n$
satisfying $(e_i,e_j)=2\delta_{ij}$ for $1 \le i,j \le n$.
Set
$\Lambda = \sum_{i=1}^n \ZZ e_i$,
$\Lambda_{\varepsilon}=\{\sum_{i=1}^n x_i e_i \mid x_i \in \ZZ,
\sum_{i=1}^n x_i \equiv \varepsilon \pmod 2\}$
($\varepsilon=0,1$).
The following lattices
are defined in~\cite{KKM}:
\begin{align*}
L_{\text{A}}(D)&=\bigcup_{x \in D}\Big(\Lambda+\frac{1}{2}e_x\Big),
\\
L_{\text{B}}(D)&=\bigcup_{x \in D}\Big(\Lambda_0+\frac{1}{2}e_x\Big),
\\
L_{\text{C}}(D)&=\bigcup_{x \in D}
\Big\{
\Big(\Lambda_0+\frac{1}{2}e_x\Big) \bigcup
\Big(\Lambda_{\varepsilon}+\frac{1}{2}e_x+\frac{1}{4}e_{\allone}\Big)
\Big\} \ (\varepsilon \equiv n/8 \pmod 2),
\end{align*}
where
%% $\alpha \equiv n/8 \pmod 2$,
$e_x=\sum_{i \in \supp(x)}e_i$ and
$\supp(x)$ denotes the support of $x$.
% and $\allone$ denotes the all-one vector.
These lattices and their
relationships with frames of Type A, B and C
are investigated in~\cite{KKM}.
For example,
if $D$ is a doubly even self-dual code of length $40$
and minimum weight $4$ (resp.\ $8$), then $L_{\text{C}}(D)$ is an
even unimodular lattice with minimum norm $2$ (resp.\ $4$)~\cite{KKM}.
In addition, the following result is an important tool
in this section.

\begin{lem}[{\cite[Theorem 3]{KKM}}]
\label{lem:KKM3}
For {\rm U} $=$ {\rm A, B, C},
a mapping
$D \mapsto L_{\text{\rm U}}(D)$ gives a one-to-one
correspondence between equivalence classes of
doubly even codes $D$ of length $n$
and isomorphism classes of even lattices in dimension $n$
with related frames of Type {\rm U},
if $n >16$  for {\rm U $=$ B}
and if $n >32$ for  {\rm U $=$ C}.
\end{lem}

As an example of the above lemma, we directly
have the following
result from the
classification of doubly even self-dual codes
of length $40$ in~\cite{BHM40}.

\begin{rem}\label{rem:E}
There are $94343$ non-isomorphic extremal even unimodular lattices
in dimension $40$  with related frames of Type {\rm C},
$16470$ of which are extremal.
\end{rem}

%%%%%
\subsection{Theta series}

\begin{lem}\label{lem:L1}
Let $L$ be an odd unimodular lattice in dimension
$n \equiv 0 \pmod 4$.
Suppose that $L$ and its shadow
$S$ have minimum norms $\ge 4$ and $2$, respectively.
Then all vectors of norm $2$ in $S$
are contained in one of $L_1$ and $L_3$.
\end{lem}
\begin{proof}
Let $x,y$ be distinct vectors of norm $2$ in $S$.
Since $x-y \in L$,
\begin{equation}\label{eq:ip}
(x-y,x-y)=4-2(x,y)\ge 4.
\end{equation}
Suppose that $x \in L_1$.
Since
$L_0^* /L_0$ is isomorphic to the Klein 4-group
(see e.g.~\cite[Lemma 1]{DHS}), $-x \in L_1$.
Hence, we may assume without loss of generality that $(x,y)\ge 0$.
If $y \in L_3$, then,
by~\cite[Lemma 2]{DHS}, $(x,y) \in \frac{1}{2}+\ZZ$.
This contradicts (\ref{eq:ip}).
\end{proof}

Conway and Sloane~\cite{CS-odd} show that
when the theta series of an odd unimodular lattice $L$
in dimension $n$ is written as
\begin{equation}
\label{Eq:theta}
\theta_L(q)=
 \sum_{j =0}^{\lfloor n/8\rfloor} a_j\theta_3(q)^{n-8j}\Delta_8(q)^j,
%%= \sum_i A_i q^i \text{ (say),}
\end{equation}
the theta series of the shadow $S$ is written as
\begin{equation}
\label{Eq:theta-S}
\theta_S(q)= \sum_{j=0}^{\lfloor n/8\rfloor}
\frac{(-1)^j}{16^j} a_j\theta_2(q)^{n-8j}\theta_4(q^2)^{8j},
%%= \sum_i B_i q^i \text{ (say)}
\end{equation}
where
$\Delta_8(q) = q \prod_{m=1}^{\infty} (1 - q^{2m-1})^8(1-q^{4m})^8$
and $\theta_2(q), \theta_3(q)$ and $\theta_4(q)$ are the Jacobi
theta series~\cite{SPLAG}.
% As the additional conditions, we have that
% there is at most one nonzero $B_r$ for $r < (\mu+2)/2$;
% $B_r=0$ for $r < \mu/4$; and $B_r \le 2$ for $r < \mu/2$
% where $\mu$ is the minimum norm of $L$.
In the case $n=40$ and minimum norm $\min(L)=4$,
$a_0, \ldots, a_3$ in (\ref{Eq:theta})
are determined as follows:
\[
a_0=1, a_1=-80, a_2=1360, a_3=-2560.
\]
In this case, it follows that
\[
\theta_S(q)=
\frac{-a_5}{2^{20}}
+ \Big(\frac{a_4}{2^8} + \frac{5a_5}{2^{16}}\Big) q^2
+ \cdots.
\]
Hence, $a_5=0$ and $a_4$ is divisible by $2^8$,
so we put $a_4=2^8\alpha$, where $\alpha$ is an integer.
Then we have the possible theta series
$\theta_{40,L,\alpha}$ and
$\theta_{40,S,\alpha}$ of an extremal odd unimodular lattice
$L$ and its shadow $S$:
\begin{align*}
\theta_{40,L,\alpha} &=
1
+ (19120  + 256 \alpha) q^4
+ (1376256 - 4096 \alpha) q^5 + \cdots,
\\
\theta_{40,S,\alpha} &=
\alpha q^2
+ (40960 - 56 \alpha) q^4
+ (87818240 + 1500 \alpha) q^6 +\cdots,
\end{align*}
respectively.
Moreover, we have the following restriction on $\alpha$.

\begin{lem}\label{lem:80}
$\alpha$ is even with $0 \le \alpha \le 80$.
\end{lem}
\begin{proof}
Denote by $S_2$ the set of all vectors of norm $2$ in $S$.
By Lemma~\ref{lem:L1}, we may assume without loss of generality
that all vectors of $S_2$
are contained in $L_1$.
Let $x,y$ be vectors of norm $2$ in $L_1$ such that
$x \ne y$ and $x \ne -y$.
If $y \in L_1$ then $-y \in L_1$.
Hence, $\alpha$ is even and
we may assume that $(x,y)\ge 0$.
It follows from (\ref{eq:ip}) that $(x,y)=0$.
The set $S_2$ is written as
\[
S_2=T \cup (-T)
\]
satisfying that
$(x,y)=0$ for
$x,y \in T$ with $x\ne y$.
Hence, $\# T \le 40$, and
$L_1$ has at most $80$ vectors
of norm $2$.
Therefore, $\alpha \le 80$.
\end{proof}

\begin{lem}\label{lem:L1L3}
Let $L$ be an extremal odd unimodular lattice in dimension
$40$ with theta series $\theta_{40,L,\alpha}$.
Then one of $L_0 \cup L_1$ and $L_0 \cup L_3$ is an
extremal even unimodular lattice
and the remaining one is an
even unimodular lattice containing $\alpha$ vectors of norm $2$.
\end{lem}
\begin{proof}
Follows from Lemma~\ref{lem:L1} and $\theta_{40,S,\alpha}$.
\end{proof}

%%%%%
\subsection{Lattices with theta series $\theta_{40,L,80}$}
Now we suppose that $L$ is an extremal odd unimodular lattice
in dimension $40$ with shadow $S$ having exactly $80$
vectors of norm $2$,
that is, $L$ and $S$ have the following
theta series:
\begin{align*}
\theta_{40,L,80} &=1 + 39600 q^4 + 1048576 q^5 + \cdots,
\\
\theta_{40,S,80} &=80 q^2 + 36480 q^4 + 87938240 q^6 +\cdots,
\end{align*}
respectively.

\begin{lem}\label{lem:L0}
There is a frame of Type {\rm B} related to
the even sublattice $L_0$.
\end{lem}
\begin{proof}
By the proof of Lemma~\ref{lem:80},
the set of all vectors of norm $2$ in $S$
may be written as
$T \cup (-T)$,
where $T=\{e_1,e_2,\ldots,e_{40}\}$ satisfying
the condition (\ref{eq:C1}).
Then
\[
\Big(\frac{1}{2} \sum_{i=1}^{40} \ZZ e_i\Big)^*
= \sum_{i=1}^{40} \ZZ e_i
\subset L \cup S= L_0^*.
\]
Hence, $e_1,e_2,\ldots,e_{40}$ satisfy
the condition (\ref{eq:CB}).
\end{proof}

We are in a position to state and prove the main result of
this section.

\begin{thm}\label{thm:odd}
There are $16470$ non-isomorphic extremal odd unimodular
lattices in dimension $40$
with theta series $\theta_{40,L,80}$.
\end{thm}
\begin{proof}
Let $L$ be an extremal odd unimodular lattice
in dimension $40$ with theta series $\theta_{40,L,80}$.
By Lemma~\ref{lem:L0},
there is a frame of Type B related to
the even sublattice $L_0$.
Hence, by Theorem 1 in~\cite{KKM}, there is a doubly even
code $C$ of length $40$ such that
$L_0 \cong L_{\text{B}}(C)$.
Since $L_{\text{B}}(C)$ is a sublattice of index $2$
of an even unimodular lattice $L_{\text{A}}(C)$,
$C$ must be self-dual.
Since $L_{\text{B}}(C)$ has minimum norm $4$,
$C$ has minimum weight $8$, that is, extremal.

Moreover, by Lemma~\ref{lem:KKM3}, a mapping
$C \mapsto L_{\text{B}}(C)$ gives a one-to-one
correspondence between equivalence classes of
extremal doubly even self-dual codes $C$ of length $40$
and isomorphism classes of even sublattices of
extremal odd unimodular lattices in
dimension $40$ with theta series $\theta_{40,L,80}$.
There are $16470$ inequivalent
extremal doubly even self-dual codes of length $40$~\cite{BHM40}.
The result follows.
\end{proof}

\begin{rem}
A similar argument can be found in~\cite{CS-odd}
for dimension $n=32$.
There is a one-to-one
correspondence between equivalence classes of
extremal doubly even self-dual codes of length $32$
and isomorphism classes of even sublattices of
extremal odd unimodular lattices in
dimension $32$ with shadows having exactly $64$
vectors of norm $2$.
In this dimension, the shadow of
any extremal odd unimodular lattice
has $64$ vectors of norm $2$.
%%%%
By Lemma~\ref{lem:L1},
the $64$ vectors of norm $2$ of the shadow
are contained in one of $L_1$ and $L_3$.
It is incorrectly reported
in~\cite[p.~360]{CS-odd} that each $L_i$ has
$32$ vectors of norm $2$ $(i=1,3)$.
\end{rem}

Now we give characterizations of the
set of vectors of minimum norm in
extremal odd unimodular lattices
with theta series $\theta_{40,L,80}$.
By the above theorem,
the even sublattice of such a lattice
is written by $L_{\text{B}}(C)$
using some extremal doubly even self-dual code $C$ of length $40$.

\begin{prop}
The set of vectors of norm $4$ in $L_{\text{\rm B}}(C)$
is given by
\[
L_{\text{\rm B}}(C)_4=
\{\pm e_i \pm e_j \mid i \ne j\}
\cup
\Big(\bigcup_{x \in C_8}
\Big\{\frac{1}{2}e_x - \sum_{y \in S}e_y\mid
S \subset \supp(x), \#S \in 2\ZZ
\Big\}
\Big),
\]
where $C_8$ is the set of codewords of weight $8$.
\end{prop}
\begin{proof}
The norm of a vector of $L_{\text{B}}(C)_4$ is $4$.
Since $\#C_8 =285$~\cite{Mallows-Sloane},
it follows that
$\#L_{\text{B}}(C)_4=2^2 \binom{40}{2}+285 \cdot 2^7 =39600$,
which is the same as the number of vectors of norm $4$
in $L_{\text{\rm B}}(C)$.
\end{proof}

\begin{rem}
By Lemmas~\ref{lem:L1} and \ref{lem:L1L3},
the even unimodular neighbor
$L_0 \cup L_3$ is extremal.
Then $L_0 \cup L_3$ is written as
$L_{\text{C}}(C)$,
where $L_0=L_{\text{\rm B}}(C)$
and $C$ is
some extremal doubly even self-dual code $C$ of length $40$
(see Remark~\ref{rem:E}).
An even unimodular lattice in dimension $40$ has the
following theta series
$1 + 39600 q^4 + 93043200 q^6 + \cdots$.
By the above proposition,
the set of vectors of norm $4$ in $L_{\text{C}}(C)$
is also given by $L_{\text{B}}(C)_4$.
\end{rem}

In Proposition~\ref{prop:Cmin},
the set of codewords of minimum weight in
an extremal singly even self-dual code of length $40$
with weight enumerator $W_{40,C,10}$ has been characterized.
Similarly, we give a characterization of
the set of vectors of minimum norm in an extremal
odd unimodular lattices in dimension $40$
with theta series $\theta_{40,L,80}$.
It is known that
the set of vectors of each norm in an extremal
even unimodular lattice in dimension $40$ forms
a spherical $3$-design
(see~\cite{BB09} for a recent survey on these subjects).
By the above remark, we have the following:

\begin{prop}\label{prop:Lmin}
The set of vectors of minimum norm in an extremal
odd unimodular lattice in dimension $40$ with theta series
$\theta_{40,L,80}$ forms a spherical $3$-design.
\end{prop}

%%%%%
\section{Extremal self-dual $\ZZ_4$-codes
and extremal odd unimodular lattices}\label{Sec:Z4}

Let $\cC$ be a self-dual $\ZZ_4$-code of length $n$ and minimum
Euclidean weight $d_E(\cC)$.
Then the following lattice
\[
A_{4}(\cC) = \frac{1}{2}
\{(x_1,\ldots,x_n) \in \ZZ^n \mid
(x_1 \bmod 4,\ldots,x_n \bmod 4)\in \cC\}
\]
is a unimodular lattice in dimension $n$ having minimum
norm $\min\{4,d_E(\cC)/4\}$.
In addition, $\cC$ is Type II if and only if
$A_4(\cC)$ is even~\cite{Z4-BSBM}.
%% This construction of lattices is called Construction A.
A set $\{f_1, \ldots, f_{n}\}$ of $n$ vectors $f_1, \ldots, f_{n}$ of a
unimodular lattice $L$ in dimension $n$ with
$(f_i, f_j) = 4 \delta_{ij}$
is called an {\em orthogonal frame of norm $4$}
(a $4$-frame for short) of $L$.
It is known that $L$ has an orthogonal frame of norm $4$ if and only
if there is a self-dual $\ZZ_4$-code $\cC$ with
$L \cong A_4(\cC)$.

\begin{lem}\label{lem:Z4}
Suppose that $n$ is even.
Let $L$ be an even (resp.\ odd) unimodular lattice in dimension $n$
such that there are vectors $e_1,e_2,\ldots,e_n$ satisfying
the condition that
$(e_i,e_j)=2\delta_{ij} \text{ and }
e_i\pm e_j \in L \ (1 \le i,j \le n)$ which is the same condition as
(\ref{eq:C1}).
Then $L$ contains an orthogonal frame of norm $4$ and
there is a Type~II (resp.\ Type~I) $\ZZ_4$-code $\cC$ of
length $n$ with $A_4(\cC) \cong L$.
\end{lem}
\begin{proof}
The following set
\[
\{e_{2i-1}+e_{2i},e_{2i-1}-e_{2i} \mid i=1,2,\ldots,n/2\}
\]
is an orthogonal frame of norm $4$ of $L$.
The result follows.
\end{proof}

%Now we consider the existence of
%extremal self-dual $\ZZ_4$-codes of length $40$.
%By Lemmas~\ref{lem:Z4} and \ref{lem:L0},
%$L_{\text{B}}(C)$ contains an orthogonal frame of
%norm $4$.
We consider the existence of
extremal self-dual $\ZZ_4$-codes.

\begin{prop}\label{prop:Z4}
There are at least  $16470$ inequivalent extremal Type~II $\ZZ_4$-codes
of length $40$.
There are at least  $16470$ inequivalent extremal Type~I $\ZZ_4$-codes
of length $40$.
\end{prop}
\begin{proof}
Since the two cases are similar, we give details only for Type~I
$\ZZ_4$-codes. By Theorem~\ref{thm:odd}, $16470$ non-isomorphic extremal
odd unimodular lattices $L_i$ $(i=1,2,\ldots,16470)$ in dimension
$40$ exist (see Remark~\ref{rem:E} for  extremal even unimodular
lattices). Moreover, for each $L_i$, there are vectors
$e_1,e_2,\ldots,e_{40}$ satisfying the condition that
$(e_j,e_k)=2\delta_{jk} \text{ and } e_j\pm e_k \in L_i \ (1 \le
j,k \le 40)$. By Lemma~\ref{lem:Z4}, there is a Type~I
$\ZZ_4$-code $\cC_i$ of length $40$ with $A_4(\cC_i) \cong L_i$
for $i=1,2,\ldots,16470$. Since $L_1, \ldots, L_{16470}$ are
non-isomorphic extremal odd unimodular lattices,
$\cC_1,\ldots,\cC_{16470}$ are inequivalent extremal Type~I
$\ZZ_4$-codes of length $40$.
%% By Remark~\ref{rem:E} (resp.\ Theorem~\ref{thm:odd}),
%% $16470$ non-isomorphic extremal
%% even (resp.\ odd) unimodular lattices
%% $L_i$ (resp.\ $L'_{i}$)
%% $(i=1,2,\ldots,16470)$ in dimension $40$ exist.
%% Moreover, for $L_i$ (resp.\ $L'_i$),
%% there are vectors $e_1,e_2,\ldots,e_{40}$ satisfying
%% the condition that
%% $(e_j,e_k)=2\delta_{jk} \text{ and }
%% e_j\pm e_k \in L_i \ (\text{resp.\ }L'_i) \ (1 \le j,k \le 40)$.
%% By Lemma~\ref{lem:Z4},
%% there is a Type~II $\ZZ_4$-code $\cC_i$
%% (resp.\ Type~I $\ZZ_4$-code $\cC'_i$) of length $40$
%% with $A_4(\cC_i) \cong L_i$
%% (resp.\ $A_4(\cC'_i) \cong L'_i$) for $i=1,2,\ldots,16470$.
%% Since $L_1, \ldots, L_{16470}$
%% (resp.\ $L'_1, \ldots, L'_{16470}$) are
%% non-isomorphic extremal even (resp.\ odd)
%% unimodular lattices,
%% $\cC_1,\ldots,\cC_{16470}$
%% (resp.\ $\cC'_1,\ldots,\cC'_{16470}$)
%% are inequivalent extremal
%% Type~II (resp.\ Type~I) $\ZZ_4$-codes of length $40$.
\end{proof}

Using the method given in~\cite{Z4-PLF},
we found $16$ new extremal Type~I $\ZZ_4$-codes
$\cD_i$ ($i=1,2,\ldots,16$) of length $40$,
in order to give extremal odd unimodular lattices $A_4(\cD_i)$
with other theta series $\theta_{40,L,4k}$, where
$k =0,1,\ldots,14,16$.
Since a Type~I $\ZZ_4$-code $\cC$
of length $40$ is extremal if and only if $A_4(\cC)$ is extremal,
we verified that odd unimodular lattices $A_4(\cD_i)$ are extremal.
%M where the minimum norms of $A_4(\cD_i)$ were
%M calculated by the {\sc Magma} function
%M {\tt Minimum}.
Also, the theta series of $A_4(\cD_i)$
were determined by obtaining
the numbers of vectors of minimum norm.
%M This was done by the {\sc Magma} function {\tt KissingNumber}.
Then we have the following:

\begin{prop}
There is an extremal odd unimodular lattice
with theta series $\theta_{40,L,4k}$
%\[\alpha=
%0, 4, 8, 12, 16, 20, 24, 28, 32, 36, 40, 44, 48, 52, 56, 64, 80.\]
for $k \in \{0,1,\ldots,14,16,20\}$.
\end{prop}

As an example,
we give a generator matrix $G$ of the Type~I $\ZZ_4$-code $\cD_1$
such that $A_4(\cD_1)$ has theta series $\theta_{40,L,0}$,
which is $s$-extremal in the sense of~\cite{Ga07}.
Since $G$ is of the form
${\displaystyle
\left(\begin{array}{ccc}
I_{10} &    \multicolumn{2}{c}{A} \\
O      & 2I_{20} &2B
\end{array}\right),
}$
we only list the matrices $A$ and $2B$ in  Figure~\ref{Fig},
to save space,
where $I_k$ denotes the identity matrix of order $k$
and $O$ denotes the $20 \times 10$ zero matrix.

\begin{figure}[htb]
\centering
%{\small
{\footnotesize
\[
A=
\left(\begin{array}{cccccccccc}
010110011110011111002023023120 \\
101011001101001111102200122110 \\
110101100100100111110220010211 \\
111010110010010011113022001201 \\
011101011011001001111322022322 \\
001110101111100100112112000032 \\
100111010111110010012033220223 \\
110011101011111001003003102020 \\
011001110101111100100100110002 \\
101100111000111110012012013202
\end{array}\right), \quad
2B=
\left(\begin{array}{cccccccccc}
2000002202 \\
2200000220 \\
0220000022 \\
2022000002 \\
2202200000 \\
0220220000 \\
0022022000 \\
0002202200 \\
0000220220 \\
0000022022 \\
2220020220 \\
0222002022 \\
2022200202 \\
2202220020 \\
0220222002 \\
2022022200 \\
0202202220 \\
0020220222 \\
2002022022 \\
2200202202
\end{array}\right)
\]
\caption{A generator matrix of the code $\cD_1$}
\label{Fig}
}
\end{figure}

% 1000000000010110011110011111002023023120 \\
% 0100000000101011001101001111102200122110 \\
% 0010000000110101100100100111110220010211 \\
% 0001000000111010110010010011113022001201 \\
% 0000100000011101011011001001111322022322 \\
% 0000010000001110101111100100112112000032 \\
% 0000001000100111010111110010012033220223 \\
% 0000000100110011101011111001003003102020 \\
% 0000000010011001110101111100100100110002 \\
% 0000000001101100111000111110012012013202 \\
% 0000000000200000000000000000002000002202 \\
% 0000000000020000000000000000002200000220 \\
% 0000000000002000000000000000000220000022 \\
% 0000000000000200000000000000002022000002 \\
% 0000000000000020000000000000002202200000 \\
% 0000000000000002000000000000000220220000 \\
% 0000000000000000200000000000000022022000 \\
% 0000000000000000020000000000000002202200 \\
% 0000000000000000002000000000000000220220 \\
% 0000000000000000000200000000000000022022 \\
% 0000000000000000000020000000002220020220 \\
% 0000000000000000000002000000000222002022 \\
% 0000000000000000000000200000002022200202 \\
% 0000000000000000000000020000002202220020 \\
% 0000000000000000000000002000000220222002 \\
% 0000000000000000000000000200002022022200 \\
% 0000000000000000000000000020000202202220 \\
% 0000000000000000000000000002000020220222 \\
% 0000000000000000000000000000202002022022 \\
% 0000000000000000000000000000022200202202

%%As a consequence,
%since equivalent codes $C$ and $D$
%give isomorphic lattices $A_4(C)$ and $A_4(D)$,
The $16$ codes $\cD_i$ ($i=1,2,\ldots,16$)
slightly
improve the number of known extremal Type~I $\ZZ_4$-codes
of length $40$ given in Proposition~\ref{prop:Z4}, that is,
there
are at least  $16486$ inequivalent extremal Type~I $\ZZ_4$-codes
of length $40$.
%%%
Generator matrices for all codes $\cD_i$
can be obtained electronically from
%\begin{verbatim}
``\verb+http://sci.kj.yamagata-u.ac.jp/~mharada/Paper/z4-40.txt+''.
%\end{verbatim}

%%%%%%%%%%%%%%%%%%%%%%
\bigskip
\noindent {\bf Acknowledgments.}
The authors would like to thank Masaaki Kitazume and
Hiroki Shimakura for useful discussions.

%%%%%%%%%%%%%%%%%%%  References  %%%%%%%%%%%%%%%%%%%%%%%%

%%%%%%%%%%%%%%%%%%%%%%%%%%%%%%%%%


\begin{thebibliography}{99}
\bibitem{n=36}
C. Aguilar-Melchor and P. Gaborit, On the classification of
extremal $[36,18,8]$ binary self-dual codes, {\sl IEEE Trans.
Inform. Theory} {\bf 54} (2008), 4743--4750.

\bibitem{AGKSS38}
C. Aguilar-Melchor, P. Gaborit, J.-L. Kim, L. Sok and P. Sol\'e,
Classification of extremal and $s$-extremal binary self-dual codes
of length 38, {\sl IEEE\ Trans.\ Inform.\ Theory}, (to appear),
arXiv:1111.0228v1.

\bibitem{A-P} E.F.~Assmus, Jr.\ and V.~Pless,
{On the covering radius of extremal self-dual codes},
{\sl IEEE\ Trans.\ Inform.\ Theory}
{\bf 29} (1983), 359--363.

\bibitem{BG00}C. Bachoc and P. Gaborit,
On extremal additive $\Bbb{F}_{4}$ codes of length 10 to 18,
{\sl J. The\'or.\ Nombres Bordeaux}
{\bf 12} (2000), 255--271.

\bibitem{BB09}E. Bannai and Ets. Bannai,
A survey on spherical designs and algebraic combinatorics on spheres,
{\sl European J. Combin.}
{\bf 30} (2009),  1392--1425.

\bibitem{BHM40} K. Betsumiya, M. Harada and A. Munemasa,
A complete classification of doubly even self-dual codes of length $40$,
(submitted), arXiv: 1104.3727.

% \bibitem{B06} \textbf{(remove)} R.T. Bilous,
% Enumeration of the binary self-dual codes of length 34,
% {\sl J. Combin.\ Math.\ Combin.\ Comput.}
% {\bf 59} (2006), 173--211.

% \bibitem{BR02} \textbf{(remove)} R.T. Bilous and G.H.J. van Rees,
% An enumeration of self-dual codes of length 32, {\sl Des.\ Codes,\
% Cryptogr.} {\bf 26} (2002), 61--86.

\bibitem{Z4-BSBM} A. Bonnecaze, P. Sol\'e, C. Bachoc and B. Mourrain,
{Type II codes over $\ZZ_4$},
{\sl IEEE\ Trans.\ Inform.\ Theory}
{\bf 43} (1997), 969--976.

\bibitem{Magma}W. Bosma, J.J. Cannon, C. Fieker and A. Steel,
{\sl Handbook of Magma Functions (Edition 2.17)}, 2010, 5117 pages.

\bibitem{BB38}S. Bouyuklieva and I. Bouyukliev,
An algorithm for classification of binary self-dual codes, {\sl
IEEE\ Trans.\ Inform.\ Theory}, (to appear), arXiv:1106.5930.

\bibitem{BY40} S. Buyuklieva and V. Yorgov, Singly-even self-dual codes of
length 40,
{\sl Des.\ Codes,\ Cryptogr.}
\textbf{9} (1996), 131--141.

\bibitem{BP} R.~Brualdi and V.~Pless,
Weight enumerators of self-dual codes,
{\sl IEEE Trans.\ Inform.\ Theory}
{\bf 37} (1991), 1222--1225.

\bibitem{CRSS}A.R. Calderbank, E.R. Rains, P.W. Shor and N.J.A. Sloane,
Quantum error correction via codes over ${\rm GF}(4)$,
{\sl IEEE Trans.\ Inform.\ Theory}
{\bf 44}  (1998),  1369--1387.

\bibitem{CPS} J.H. Conway, V. Pless and N.J.A. Sloane,
{The binary self-dual codes of length up to $32$: a revised enumeration},
{\sl J.\ Combin. Theory Ser.~A}
{\bf 60} (1992), 183--195.

\bibitem{C-S} J.H.~Conway and N.J.A.~Sloane,
A new upper bound on the minimal distance of self-dual codes,
{\sl IEEE\ Trans.\ Inform.\ Theory}
{\bf 36} (1990), 1319--1333.

\bibitem{CS-odd}J.H.~Conway and N.J.A.~Sloane,
{A note on optimal unimodular lattices},
{\sl J.\ Number Theory}
{\bf 72} (1998), 357--362.

\bibitem{SPLAG} J.H.~Conway and N.J.A.~Sloane,
{\sl Sphere Packing, Lattices and Groups (3rd ed.)},
Springer-Verlag, New York, 1999.

\bibitem{DP06}L.E. Danielsen and M.G. Parker,
On the classification of all self-dual additive codes
over ${\rm GF}(4)$ of length up to 12,
{\sl J. Combin.\ Theory Ser.~A}
{\bf 113} (2006), 1351--1367.

\bibitem{DHS}S.T. Dougherty, M. Harada and P. Sol\'e,
Shadow lattices and shadow codes,
{\sl Disc.\ Math.}
{\bf 219} (2000), 49--64.

\bibitem{Ga07}P. Gaborit,
A bound for certain $s$-extremal lattices and codes,
{\sl Arch.\ Math.}
{\bf 89} (2007), 143--151.

% \bibitem{Gaulter}M.~Gaulter,
% {Minima of odd unimodular lattices in dimension $24m$},
% {\sl J. Number Theory}
% {\bf 91} (2001), 81--91.

\bibitem{HM06} M. Harada and A. Munemasa,
Some restrictions on weight enumerators of singly even self-dual codes,
{\sl IEEE Trans.\ Inform.\ Theory}
{\bf 52} (2006), 1266--1269.

\bibitem{HM36} M. Harada and A. Munemasa,
Classification of self-dual codes of length $36$,
{\sl Adv.\ Math.\ Commun.},
(to appear), arXiv:1012.5464.

\bibitem{Data} M. Harada and A. Munemasa,
{Database of Self-Dual Codes},
Available online at
{\verb+ http://www.math.is.tohoku.ac.jp/~munemasa/+}
{\verb+ selfdualcodes.htm+}.

\bibitem{HO}M. Harada and M. Ozeki,
Extremal self-dual codes with the smallest covering radius,
{\sl Disc.\ Math.}
{\bf 215}  (2000),  271--281.

\bibitem{Z4-HSG} M. Harada, P. Sol\'e and  P. Gaborit,
{Self-dual codes over $\ZZ_4$ and unimodular lattices: a survey},
Algebras and combinatorics (Hong Kong, 1997),
255--275, Springer, Singapore, 1999.

\bibitem{Hohn}G. H\"ohn,
Self-dual codes over the Kleinian four group,
{\sl Math.\ Ann.}
{\bf 327}  (2003),  227--255.

\bibitem{KKM}M. Kitazume, T. Kondo and I. Miyamoto,
Even lattices and doubly even codes,
{\sl J. Math.\ Soc.\ Japan}
{\bf  43} (1991),  67--87.

\bibitem{McK}
B. D. McKay, {Isomorph-free exhaustive generation},
{\sl J. Algorithms} {\bf 26} (1998), 306--324.

\bibitem{Mallows-Sloane} C.L.~Mallows and N.J.A.~Sloane,
{An upper bound for self-dual codes}, {\sl Inform.\ Control} {\bf
22} (1973), 188--200.

\bibitem{Ozeki-S} M.~Ozeki,
{Jacobi polynomials for singly even self-dual codes and the
covering radius problems},
{\sl IEEE\ Trans.\ Inform.\ Theory}
{\bf 48} (2002), 547--557.

% \bibitem{Pless72} \textbf{(remove)} V. Pless,
% A classification of self-orthogonal codes over ${\rm GF}(2)$, {\sl
% Disc.\ Math.} {\bf  3}  (1972), 209--246.

\bibitem{Z4-PLF} V. Pless, J. Leon and J. Fields,
{All $\ZZ_4$ codes of Type~II and length 16 are known},
{\sl J.\ Combin.\ Theory\ Ser.~A}
{\bf 78} (1997), 32--50.

\bibitem{PS75}V. Pless and N.J.A. Sloane,
On the classification and enumeration of self-dual codes,
{\sl J.\ Combin. Theory Ser.~A}
{\bf 18}  (1975), 313--335.

\bibitem{Rains} E.M.~Rains,
{Shadow bounds for self-dual codes},
{\sl IEEE Trans.\ Inform.\ Theory}
{\bf 44} (1998), 134--139.

\bibitem{RS-bound} E.~Rains and N.J.A.~Sloane,
{The shadow theory of modular and unimodular lattices},
{\sl J.\ Number Theory}
{\bf 73} (1998), 359--389.

\bibitem{RS-Handbook} E.\ Rains and N.J.A.\ Sloane,
{``Self-dual codes,''} {Handbook of Coding Theory},
V.S. Pless and W.C. Huffman (Editors),
Elsevier, Amsterdam 1998, pp.\ 177--294.

\bibitem{S}H. Shimakura,
On isomorphism problems for vertex operator algebras associated
with even lattices,
{\sl Proc.\ Amer.\ Math.\ Soc.},
(to appear), arXiv: 1104.0714.

\bibitem{Yorus} V.Y. Yorgov, Binary self-dual codes with automorphisms of odd order,
{\sl Problems Inform.\ Transmission} {\bf 19}  (1983),  260--270;
translated from {\sl Probl. Peredachi Inf.} {\bf 19}  (1983),
11--24 (Russian).


\end{thebibliography}
\end{document}